\documentclass[11pt,reqno]{scrartcl}

\usepackage[utf8]{inputenc}

\usepackage[final,nomicrotype,showeqnr,theoremdefs]{allergy}
\usepackage{unixode}
\usepackage[seccol]{helvetic}

\graphicspath{{}}

\ifxetex
\RequirePackage{fontspec}
\else
\usepackage{inconsolata}
\fi

\usepackage{pythonhighlight}

\usepackage{enumitem}

\usepackage[font=footnotesize]{subfig}

\usepackage{bm}

\usepackage{tikz}
\usetikzlibrary{arrows}
\usetikzlibrary{arrows.meta}
\usetikzlibrary{decorations.markings}
\tikzset{interpolation point/.style={circle, draw, fill=black, scale=.8}}
\tikzset{control point/.style={circle, draw, scale=.6}}
\tikzset{calc point/.style={circle, draw, fill=gray, scale=.6}}
\newcommand*\decasteljautime{0.25}
\tikzset{intnode/.style={control point, scale=.7, pos=\decasteljautime, solid, thin, fill=blue}}
\tikzset{curve/.style={very thick,semitransparent,color=red}}
\tikzset{calc/.style={thick, densely dashed}}
\tikzset{geodesic/.style={thick, densely dotted}}
\tikzset{nogeodesic/.style={thick, color=gray}}
\usetikzlibrary{calc}

\usepackage{algpseudocode}
\usepackage{algorithm}

\usepackage[numbers]{natbib}
\usepackage{booktabs}
\DeclareRobustCommand{\SkipTocEntry}[5]{}

\usepackage[affil-sl]{authblk} 

\DeclareMathOperator{\rota}{rot}

\providecommand{\norm}[1]{\lVert#1\rVert}
\providecommand{\abs}[1]{\lvert#1\rvert}

\newcommand{\from}{\vcentcolon}

\newcommand{\ud}{\mathrm{d}}

\newcommand{\RR}{{\mathbb R}}
\newcommand{\CC}{{\mathbb C}}

\DeclareMathOperator{\ad}{ad}
\DeclareMathOperator{\Ad}{Ad}
\newcommand*\GL{\mathrm{GL}}
\newcommand*\id{\mathrm{id}}

\newcommand*\SO{\mathrm{SO}}
\newcommand*\OO{\mathrm{O}}

\newcommand*\SU{\mathrm{SU}}
\newcommand*\UU{\mathrm{U}}

\newcommand*\SE{\mathrm{SE}}

\newcommand{\lie}[1]{\mathfrak{#1}}

\DeclareMathOperator {\dexp}{dexp}

\newtheorem{problem}[theorem]{Problem}

\newcommand*\act{\cdot}
\newcommand*\wel{w}
\newcommand*\vel{v}

\newcommand*\bvec[1]{\boldsymbol{#1}}

\newcommand*\CP[1]{\CC P^{#1}}
\newcommand*\Sphere[1]{S^{#1}}
\newcommand*\Tan{T} 

\title%
{A Numerical Algorithm for \texorpdfstring{$C^2$}{C2}-splines \\ on Symmetric Spaces}

\author[1]{Klas Modin\thanks{\emaillink{klas.modin@chalmers.se}}}
\author[1,2]{Geir Bogfjellmo\thanks{\emaillink{geir.bogfjellmo@icmat.es}}}
\author[3]{Olivier Verdier\thanks{\emaillink{olivier.verdier@hvl.no}}}

\affil[1]{
  Mathematical Sciences, Chalmers and University of Gothenburg, Sweden
}
\affil[2]{Instituto de Ciencias Matemática, Madrid, Spain}
\affil[3]{
  Department of Computing, Mathematics and Physics, Western Norway University of Applied Sciences, Bergen, Norway
}

\date{\today}                                           

\begin{document}

\maketitle

\begin{abstract}
	Cubic spline interpolation on Euclidean space is a standard topic in numerical analysis, with countless applications in science and technology.
	In several emerging fields, for example computer vision and quantum control, there is a growing need for spline interpolation on curved, non-Euclidean space.
	The generalization of cubic splines to manifolds is not self-evident, with several distinct approaches.
	One possibility is to mimic the acceleration minimizing property, which leads to \emph{Riemannian cubics}.
	This, however, requires the solution of a coupled set of non-linear boundary value problems that cannot be integrated explicitly, even if formulae for geodesics are available.
	Another possibility is to mimic De~Casteljau's algorithm, which leads to \emph{generalized Bézier curves}.
	To construct $C^2$-splines from such curves is a complicated non-linear problem, until now lacking numerical methods.
	Here we provide an iterative algorithm for $C^2$-splines on Riemannian symmetric spaces, and we prove convergence of linear order.
	In terms of numerical tractability and computational efficiency, the new method surpasses those based on Riemannian cubics.
	Each iteration is parallel, thus suitable for multi-core implementation.
	We demonstrate the algorithm for three geometries of interest: the $n$-sphere, complex projective space, and the real Grassmannian.

  \textbf{MSC-2010:} 41A15, 65D07, 65D05, 53C35, 53B20, 14M17

  \textbf{Keywords:} De~Casteljau, cubic spline, Riemannian symmetric space, Bézier curve
\end{abstract}

\tableofcontents

\section{Introduction}

We address the following.
\begin{problem}\label{prob:interpolation_on_M}
	Given a set of points $\set{ p_i }_{i=0}^{N}$ on a connected manifold~$M$, construct a $C^2$ path $\gamma\colon [0,N]\to M$ such that $\gamma(i) = p_i$.
\end{problem}
For $M=\RR^n$ every textbook in numerical analysis teaches \emph{cubic splines}, i.e., piecewise polynomials of order 3 with matching first and second derivatives.
However, the generalization of cubic splines to curved space is non-trivial, essentially because polynomials are not well-defined on manifolds.
Interpolating paths on manifolds are, nevertheless, needed in a growing number of applications.
In the realization of quantum computers, for example, quantum control of qubits leads to an interpolation problem in complex projective space $\CC P^n$ (see~\autoref{sec:quantumexample}).
In computer vision, as another example, the recognition of a point cloud up to affine transformations is naturally identified with an element in the Grassmannian $\mathrm{\mathrm{Gr}}(k,n)$ (see~\autoref{sec:visionexample}).

If $M$ is a Riemannian manifold, a natural generalization of cubic splines are \emph{Riemannian cubics} \cite{NoHePa89, CrSi91}.
They are based on minimizing acceleration under interpolation constraints, by solving the problem
\begin{equation}
	\min_{\gamma(i)=p_i} \int_0^N \abs{\nabla_{\dot\gamma(t)}\dot\gamma(t)}^2\, \ud t.
\end{equation}
The corresponding Euler--Lagrange equation is a fourth order ODE on $M$ with complicated, coupled boundary conditions \cite{NoHePa89}.
In addition to the Euclidean case, explicit solutions to the initial-value problem are known for so called \emph{null Lie quadratics} on $\SO(3)$ and $\SO(2,1)$ equipped with the bi-invariant Riemannian metric \cite{No2006}.
For other cases, even when geodesics are explicitly known, one is left with numerical ODE methods, combined, for example, with a shooting method to match the boundary conditions.
This approach is computationally costly and the associated convergence analysis is involved.
Another possibility is to only approximatelly fulfill the interpolation constraints by incorporating them in the minimization functional and then use a steepest-descent method on the space of curves~\cite{SaAbSrKl2011}.
This approach is also costly.

If $M$ is a manifold for which the geodesic boundary problem can be solved explicitly, there is a natural concept of \emph{generalized Bézier curves} (GBC).
They are based on a direct generalization of De~Casteljau's algorithm~\cite{PaRa1995}.
Splines can then be constructed by gluing piecewise GBC with matching boundary conditions.
In this way, a fully explicit algorithm for $C^1$-splines on Grassmannian and Stiefel manifolds has been developed~\cite{KrMaLeBa2017}.
The $C^2$~condition, however, first derived by \citet{PoNo2007}, is significantly more complicated than the $C^1$~condition.
For this reason, there is a lack of numerical algorithms for $C^2$-splines.
Nevertheless, $C^2$ continuity is often required in applications, to avoid discontinuities in accelerations.

In this paper we give the first numerical algorithm for $C^2$-splines based on GBC.
Instead of treating all Riemannian manifolds, we focus on the subclass of \emph{symmetric spaces} (see \autoref{sub:symmetric_spaces}).
Such manifolds are of high interest in applications, and include, e.g., spheres, hyperbolic space, real and complex projective spaces, and Grassmannian manifolds (see \autoref{tbl:sym_spaces}).
In many cases, the geodesic boundary value problem on a symmetric space has an explicit solution (a prerequisite for De~Casteljau's algorithm).
The original $C^2$~condition of Popiel and Noakes is, however, still complicated.
We now list the specific contributions of our paper. 

\begin{enumerate}
	\item Using the special homogeneous space structure of symmetric spaces, we give a significant simplification of the $C^2$~condition (\autoref{thm:main_result_C2}).
	\item Using the new $C^2$~condition, we construct an algorithm based on fixed-point iterations (\autoref{alg:mainalgorithm}).
	\item We prove convergence of the algorithm under conditions on the maximal distance between consecutive interpolation points (\autoref{thm:main_result_convergence}).
\end{enumerate}



The paper is organized as follows.
In \autoref{sec:gen_bez} we first give a brief description of Riemannian symmetric spaces and thereafter describe the construction of generalized Bézier curves.
We also formulate the $C^2$~condition.
In \autoref{sec:num_alg} we give the numerical algorithm and we formulate the convergence result.
Numerical examples are given in \autoref{sec:examples}.
An outlook towards future work is given in \autoref{sec:outlook}.
The proofs of the $C^2$~condition and the convergence result are long and technical, and therefore given in \autoref{sec: lemmata}-\ref{sec: conv}.
In \autoref{sec:python} we give a brief demonstration of an easy-to-use open-source \texttt{Python} package implementing our interpolation algorithm for several geometries.



\section{Generalized Bézier curves on symmetric spaces}\label{sec:gen_bez}


Let us describe the construction of generalized Bézier curves on a symmetric space.
The construction relies on a notion of \emph{interpolation} between two points on the manifold, generalizing the notion of a \emph{geodesic} on a Riemannian manifold~\cite{Lee97}.


\begin{table}
  \small
  \begin{tabular}{rlllc}
    \toprule
     & $M$ & $G$ & $H$ & Implementation \\
    \midrule
    Euclidean space & $\RR^n$ & $\SE(n)$ & $\SO(n)$ & \texttt{Flat} \\
    $n$-sphere & $S^n$ & $\OO(n+1)$ & $\OO(n)$ & \texttt{Sphere} \\
    Hyperbolic space & $H^n$ & $\OO(n,1)$ & $\OO(n)$ & \texttt{Hyperbolic} \\
    Real Grassmannian & $\mathrm{\mathrm{Gr}}(k,n)$ & $\OO(n)$ & $\OO(k)\times \OO(n-k)$ & \texttt{Grassmannian} \\
    Orthogonal group & $\OO(n)$ & $\OO(n)\times\OO(n)$ & $\OO(n)$ & \\
    Complex projective space & $\CC P^n$ & $\UU(n+1)$ & $\UU(n)\times \UU(1)$ & \texttt{Projective} \\
    Complex Grassmannian & $\mathrm{\mathrm{Gr}}_\CC(k,n)$ & $\UU(n)$ & $\UU(k)\times \UU(n-k)$ & \\
    Unitary group & $\UU(n)$ & $\UU(n)\times\UU(n)$ & $\UU(n)$ & \\
    Co-variance matrices & $P(n)$ & $\GL(n)$ & $\OO(n)$ & \\
    \bottomrule
  \end{tabular}
  \caption{Some examples of Riemannian symmetric spaces.
    For the details of the corresponding symmetric space decomposition, see~\cite[\S\,5]{MKVe16}.
    In the right-most column, we give the geometry class to use when interpolating with our \texttt{Python} package; we refer to \autoref{sec:python} for example usage.
  }\label{tbl:sym_spaces}
\end{table}

\subsection{Symmetric spaces}
\label{sub:symmetric_spaces}
In this section we briefly recall symmetric spaces.
A list of the most common examples of symmetric spaces is given in \autoref{tbl:sym_spaces}.
For details, we refer to~\cite{Sh97} or~\cite{He62}.
We assume that a Lie group $G$ acts transitively on $M$, making $M$ a \emph{homogeneous space}~\cite[\S\,4]{Sh97}.
We denote the \emph{action} by
\begin{equation}
  G\times M \ni (g, p) \mapsto g \act p \in M
  .
\end{equation}
For every point $p\in M$, we define the \emph{isotropy subgroup} $H_p$ at $p$
\[
  H_p \coloneqq \setc{g\in G}{g\act p = p}
  .
\]

Now, the homogeneous space $M$ is called a \emph{symmetric space} if the Lie algebra $\lie{g}$ of $G$ has a direct sum decomposition
\begin{equation}
  \label{eq:reddec}
  \lie{g} = \lie{h}_p \oplus \lie{m}_p
  ,
\end{equation}
where $\lie{h}_p$ is the Lie algebra of $H_p$, and the summands fulfill the following algebraic conditions
\begin{equation}
[\lie{h}_p,\lie{h}_p] \subset \lie{h}_p, \qquad [\lie{h}_p, \lie{m}_p] \subset \lie{m}_p, \qquad [\lie{m}_p, \lie{m}_p]\subset \lie{h}_p .
\label{eq: algcond}
\end{equation}
If these conditions are satisfied at one point $p\in M$, then they are automatically valid at every point with
\[
\lie{h}_{g\act p} = g \act \lie{h}_p, \qquad \lie{m}_{g\act p} = g\act \lie{m}_p ,
\]
where we follow the convention formulated in \autoref{sec: lemmata} that $g\act$ denotes the canonical action on the space at hand (here, the adjoint action of a Lie group on its Lie algebra).

By differentiating the action map $(g, p) \mapsto g\cdot p$ with respect to the arguments,
we also obtain, on the one hand, an \emph{infinitesimal action} of $\lie{g}$ on $M$,
and, on the other hand, a \emph{prolonged action} of $G$ on $\Tan M$.
For $\xi \in \lie{g}$ and $p \in M$, we denote the infinitesimal action by
\begin{equation}\label{eq:inf_action}
  (\xi,p)\mapsto\xi \act p \in \Tan_p M
  .
\end{equation}
For a fixed $p$, this map induces an isomorphism between $\lie{m}_p$ and $T_pM$.

\subsection{Interpolating Curves}

We now define the \emph{interpolating curve} $\eta\colon [0,1] \to M$ between two (close enough) points $p_0$ and $p_1$ as the unique curve satisfying
\begin{equation}
\label{eq:interp_curve}
  \eta(t) = \exp(t\xi)\cdot p_0
  \qquad
  \eta(1) = p_1
  \qquad \xi\in \lie{m}_{p_0}
  .
\end{equation}
Let us introduce the notation
\[
  [p_0,p_1]_t \coloneqq \eta(t)
  .
\]
In particular, notice that
\begin{align}
  [p_0, p_1]_0 &= p_0
                 ,
  \qquad
  [p_0, p_1]_1 = p_1
                 .
\end{align}

In many (but not all) cases, a symmetric space $M$ can be equipped with the structure of a Riemannian manifold, given through an $H_p$-invariant inner product $\langle \cdot,\cdot\rangle_p$ on $\lie{m}_p$.
In this case, the interpolating curves are geodesics.
From here on, we shall assume that $M$ is such a \emph{Riemannian symmetric space}.

\subsection{De~Casteljau's construction}\label{sub:casteljau}

\newcommand*\placepoints{

  \node (p0) at (-4, -2) [interpolation point] {};
  \node at (p0) [left=.2] {$p_i$};
  \node (p1) at (4,-2) [interpolation point] {};
  \node at (p1) [right=.2] {$p_{i+1}$};
  \node (q0) at (-4.5,1.2) [control point, label={$q_{i}^+$}] {};
  \node (q1) at (2,1) [control point, label={$q_{i+1}^-$}] {};
}

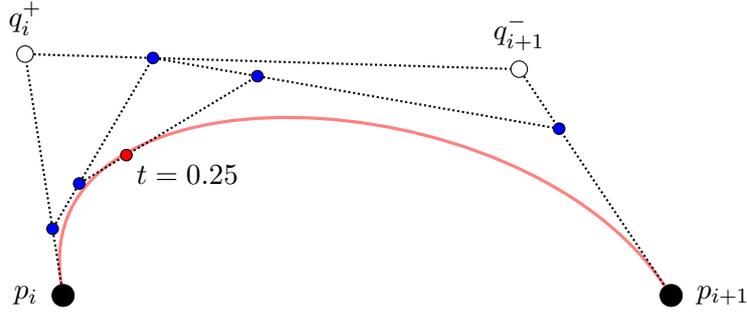
\begin{figure}
  \centering
\begin{tikzpicture}

  \placepoints

  \draw[curve] (p0)..controls (q0) and (q1) .. (p1);

  \draw[geodesic] (p0) -- (q0) node[intnode] (p0q0) {};
  \draw[geodesic] (q0) -- (q1) node[intnode] (q0q1) {};
  \draw[geodesic] (q1) -- (p1) node[intnode] (q1p1) {};

  \draw[geodesic] (p0q0) -- (q0q1) node[intnode] (p0q1) {};
  \draw[geodesic] (q0q1) -- (q1p1) node[intnode] (q0p1) {};

  \draw[geodesic] (p0q1) -- (q0p1) node[intnode, fill=red] (p0p1) {};

  \node at (p0p1) [anchor=north west] {$t = \decasteljautime$};

\end{tikzpicture}
\caption{
  An illustration of the construction of a cubic Bézier curve using De~Casteljau's algorithm
  at $t=0.25$.
  The algorithm proceeds iteratively by choosing points at position $0.25$ on each geodesic as indicated in the figure.
  The precise definition of the resulting curve is given in \eqref{eq:cubicdecasteljau}.
  }
\label{fig:decasteljau}
\end{figure}

We briefly describe the De~Casteljau construction, illustrated on \autoref{fig:decasteljau}.
This algorithm constructs Bézier curves from interpolating curves.
We focus on \emph{cubic} Bézier curves.
From four points \(p_0\), \(q_0\), \(q_1\) and \(p_1\), one defines the cubic Bézier curve $\gamma \colon [0,1] \to M$ by
\begin{equation}
\label{eq:cubicdecasteljau}
  \gamma(t) \coloneqq \bracket[\Big]{\bracket[\big]{\bracket{p_0,q_0}_t, \bracket{q_0,q_1}_t}_t, \bracket[\big]{\bracket{q_0,q_1}_t, \bracket{q_0,p_1}_t}_t}_t
  .
\end{equation}
Note that, by construction, $\gamma(0) = p_0$ and $\gamma(1) = p_1$.

\subsection{Exponential and logarithm on symmetric spaces}
\label{sec:explog}

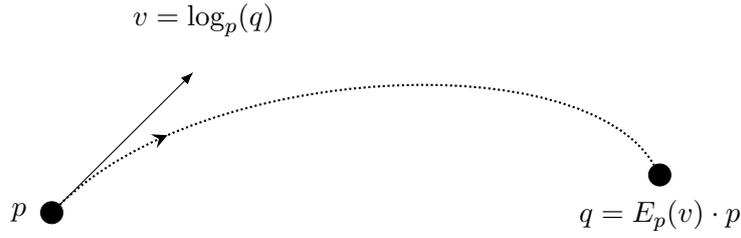
\begin{figure}
  \centering
\begin{tikzpicture}

  \node (p0) at (-4, 0) [interpolation point] {};
  \node at (p0) [left=.2] {$p$};
  \node (q0) at (4,.5) [interpolation point] {};
  \node at (q0) [below=.2] {$q = E_p(v) \act p$};
  \node (v) at (-2,2) {};
  \draw[-Latex] (p0) -- (v);

 \node at (v) [above=.2] {$v = \log_p(q)$};

  \begin{scope}[
    decoration={
      markings,
      mark=at position 0.2 with {\arrow[line width=.5mm]{>}}
    },
    ]
    \draw[postaction={decorate}, geodesic] (p0) .. controls (v) and +(120:2cm) .. (q0)
     {};

  \end{scope}

\end{tikzpicture}
\caption{
  Illustration of the movement function $E_p$ and of the logarithm described in \autoref{sec:explog}.
  The curve between $p$ and $q$ is a geodesic.
}
\label{fig:explog}
\end{figure}

For the $C^2$ condition in \autoref{sec:c2condition}
and implementation of the algorithm
we need the following functions (see \autoref{fig:explog}).
\begin{enumerate}
\item
  The action of a group element $g \in G$ on a point $x\in M$:
  \begin{equation}\label{eq:action_map}
    g \act x \in M .
  \end{equation}
\item
  A movement function $E$,
  which,
  to a point $p$ and a velocity $v$ at $p$,
  assigns an element $E_p(v) \in G$.
  This function should generate interpolating curves \eqref{eq:interp_curve}, i.e., it should be of the form
  \begin{equation}
    \label{eqdef:movement}
    E_p(\xi \cdot p) = \exp(\xi) \qquad \xi \in \lie{m}_p
    .
  \end{equation}
\item
The \emph{logarithm}; given two points $p$ and $q$ in $M$, it is defined by
\begin{equation}
  \label{eqdef:log}
  v \in \Tan_p M, \qquad E_p(v)\act p = q, \qquad \iff \qquad v = \log_p(q)
  .
\end{equation}
Throughout the paper we make the blanket assumption that $p$ and $q$ are never past conjugate points, so that the logarithm is always well-defined.
\end{enumerate}

\subsection{Cubic splines and \texorpdfstring{$C^2$}{C2} condition}\label{sec:c2condition}

We shall now consider curves consisting of piecewise cubic Bézier curves.
As we have seen in \autoref{sub:casteljau}, each cubic Bézier curve is determined by two interpolating points and two control points.
To construct a curve of piecewise Bézier curves we therefore need to impose conditions on the control points (in addition to the interpolating conditions).
On Euclidean space, the standard approach is to use the conditions for $C^2$ continuity at the interpolating points, which leads to a linear set of equations.
On a general Riemannian manifold, the corresponding $C^2$ condition, given by \citet{PoNo2007}, is highly nonlinear, involving the inverse of the derivative of the Riemannian exponential.
In the special case of symmetric spaces, however, we now show that the $C^2$ condition simplifies significantly, involving only the three operations \eqref{eq:action_map}, \eqref{eqdef:movement}, and \eqref{eqdef:log}.

\begin{definition}
\label{def:piecewise_cubic_bezier}
  Let $\bvec{p}=(p_0,\ldots,p_N)$ be points on a symmetric space $M$.
  A \emph{composite cubic Bézier curve} on $M$ is a curve $\gamma\colon[0,N]\to M$ such that
  for each integer $0\leq i \leq N$, $\gamma(i) = p_i$,
  and such that for each integer $0 \leq i < N$,
  the restriction $\gamma_i \coloneqq \gamma|_{[i,i+1]}$ is a cubic Bézier curve (as defined in \autoref{sub:casteljau}).
  To each such piece $\gamma_i$, we associate two \emph{control points}, denoted $q_i^+$ and $q_{i+1}^-$.
  A $C^2$-continuous composite cubic Bézier curve is called a \emph{cubic spline}.
\end{definition}

The interpolating conditions ensure that $\gamma$ is continuous (or $C^0$).
However, for higher degree of continuity ($C^1$ or $C^2$), we need additional conditions on the control points. 
We now give such conditions.



\newcommand*\placepointstwo{

  \node (p0) at (-6, -3) [interpolation point] {};
  \node at (p0) [left=.2] {$p_{i-1}$};

  \node (q0p) at (-6.5,-0.25) [control point, label={$q_{i-1}^+$}] {};

  \node (p1) at (0,0) [interpolation point] {};
  \node at (p1) [above=.1] {$p_{i}$};

  \node (q2m) at (4,-3) [control point] {};
  \node at (q2m) [below=.1] {$q_{i+1}^-$};

  \node (p2) at (6,-3) [interpolation point] {};
  \node at (p2) [right=.2] {$p_{i+1}$};

  \coordinate (delta) at ($(q2m)-(q0p)$);
  \coordinate (xi) at ($(0,0)!0.25!(delta)$);
  \coordinate (q1m) at ($(p1)-(xi)$);
  \coordinate (q1p) at ($(p1)+(xi)$);

  \node at (q1m) [control point, label={$q_{i}^-$}] {};
  \node at (q1p) [control point, label={$q_{i}^+$}] {};

  \node (q0ptrans) at ($(q0p) - (q1m)+(p1)$) [calc point] {};
  \node at (q0ptrans) [below=.1] {$\exp(\xi_i)\cdot q_{i-1}^+$};

  \node (q2mtrans) at ($(q2m) - (q1p)+(p1)$) [calc point] {};
  \node at (q2mtrans) [below=.1] {$\exp(-\xi_i)\cdot q_{i+1}^-$};

}

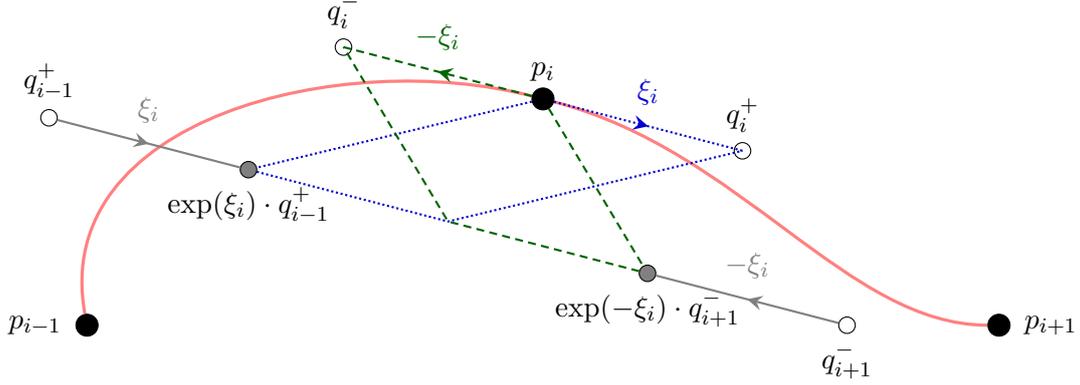
\begin{figure}
  \centering
\begin{tikzpicture}

  \placepointstwo

  \draw[curve] (p0)..controls (q0p) and (q1m) .. (p1) ..controls (q1p) and (q2m) .. (p2);

  \begin{scope}[
    decoration={
      markings,
      mark=at position 0.5 with {\arrow[line width=.5mm]{>}}
    },
    ]
    \draw[geodesic, densely dotted, color=MediumBlue, postaction={decorate}] (p1) -- (q1p)
    node[midway, label={above:$\xi_{i}$}] (v0) {};
    \draw[geodesic, densely dashed, color=DarkGreen, postaction={decorate}] (p1) -- (q1m)
    node[midway, label={above:$-\xi_{i}$}] (v1) {};
    \draw[geodesic, densely dotted, color=MediumBlue] (p1) -- (q0ptrans);
    \draw[geodesic, densely dashed, color=DarkGreen] (p1) -- (q2mtrans);
    \draw[calc, densely dotted, color=MediumBlue] (q0ptrans) -- ($(q0ptrans)+(xi)$) -- (q1p);
    \draw[calc, densely dashed, color=DarkGreen] (q2mtrans) -- ($(q2mtrans)-(xi)$) -- (q1m);
    \draw[nogeodesic, postaction={decorate}] (q0p) -- (q0ptrans)
    node[midway, label={above:$\xi_{i}$}] (v2) {};
    \draw[nogeodesic, postaction={decorate}] (q2m) -- (q2mtrans)
    node[midway, label={above:$-\xi_{i}$}] (v3) {};

  \end{scope}






\end{tikzpicture}
\caption{
  An illustration of the $C^2$~condition in \autoref{thm:main_result_C2}.
  The sum of the two vectors at $p_i$ generating the dashed green geodesics should equal the sum of the two vectors at $p_i$ generating the dotted blue geodesics.
  Notice that the action of $\exp(\xi_i)$ on $q^+_{i-1}$ and $q^-_{i+1}$ does not generate geodesics from these points, since $\xi_i$ does not belong to $\lie{m}_{q^+_{i-1}}$ or $\lie{m}_{q^-_{i+1}}$.
  For this reason, the condition \eqref{eq:c2_condition} in \autoref{thm:main_result_C2} does not work for arbitrary Riemannian manifolds.
  }
\label{fig:c2condition}
\end{figure}

\begin{theorem}\label{thm:main_result_C2}
Let $M$ be a Riemannian symmetric space and let $\gamma\colon [0,N]\to M$ be a composite cubic Bézier curve according to \autoref{def:piecewise_cubic_bezier}.
Consider the conditions on the control points given by
\begin{align}
  \log_{p_i}(q_{i}^+) &= -\log_{p_i}(q_{i}^-) \eqqcolon v_i \label{eq:c1_condition} \\
  \log_{p_i}(E_{p_i}(-\vel_i)\cdot q_{i+1}^-) - \log_{p_i}(q_{i}^+) &=  \log_{p_i}(E_{p_i}(\vel_i)\cdot q_{i-1}^+) -\log_{p_i}(q_{i}^-)
                                                                      .
  \label{eq:c2_condition}
\end{align}
\begin{enumerate}[label=\textup{(\roman*)}]
  \item If condition \eqref{eq:c1_condition} is fulfilled, then $\gamma$ is a $C^1$-curve.
  \item If conditions \eqref{eq:c1_condition} and \eqref{eq:c2_condition} are fulfilled, then $\gamma$ is a $C^2$-curve, i.e., a cubic spline.
\end{enumerate}
Furthermore, if $\gamma$ fulfills the $C^2$~conditions \eqref{eq:c1_condition} and \eqref{eq:c2_condition}, then it is uniquely determined either by
clamped boundary conditions, where $\gamma'(0)$ and $\gamma'(N)$ are prescribed,
or by natural boundary conditions, where $\nabla_{\gamma'(0)}\gamma'(0) = \nabla_{\gamma'(0)}\gamma'(0) = 0$.
\end{theorem}

\begin{remark}
  Notice how \eqref{eq:c1_condition} and \eqref{eq:c2_condition} are generalizations of the first and second order finite differences in Euclidean geometry.
  Indeed, in this case condition \eqref{eq:c1_condition} reads
  \begin{equation}
    q_i^+ - p_i = -(q_i^--p_i)
  \end{equation}
  and condition \eqref{eq:c2_condition} reads
  \begin{equation}
    q_{i-i}^+ - 2 q_i^- + p_i = q_{i+1}^- - 2 q_i^+ + p_i .
  \end{equation}
  See \autoref{fig:c2condition} for a geometric illustration of the $C^2$~condition \eqref{eq:c2_condition}.
\end{remark}

\begin{proof}[Outline of proof]
  The details of the proof are rather technical and therefore given in \autoref{sec:C2reg}.
  However, for convenience we also provide an outline here:
  \begin{enumerate}[label=\textup{(\roman*)}]
    \item
     The $C^2$ condition in the case of Riemannian symmetric spaces is given by \citet[Sec.~4]{PoNo2007}, expressed in terms of the derivative of the \emph{symmetry function} $I_p(q) = E_p(-\log_p(q))$.
    \item
     Using the special symmetric space structure, we show that the derivative of $I_p$ can be expressed solely in terms of the movement function $E_p$ (\autoref{prop:involution}).
    \item 
     Together with a result on equivariance of the $\log$ function (\autoref{prop:logequi}) we can then simplify the original Riemannian symmetric space $C^2$ condition by \citet{PoNo2007} to the one given by \eqref{eq:c2_condition}.
  \end{enumerate}
\end{proof}

\section{Numerical algorithm}\label{sec:num_alg}
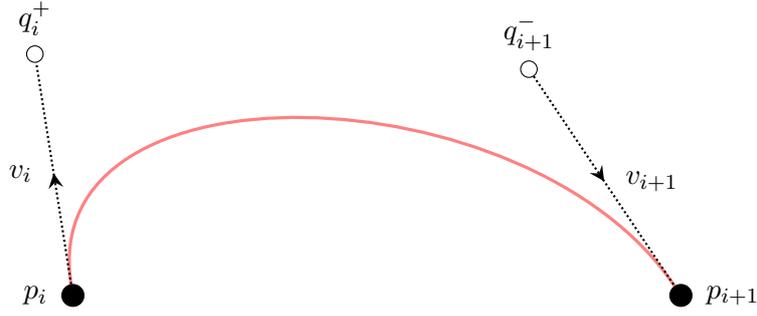
\begin{figure}
  \centering
\begin{tikzpicture}

  \placepoints

  \draw[curve] (p0)..controls (q0) and (q1) .. (p1);

  \begin{scope}[
    decoration={
      markings,
      mark=at position 0.5 with {\arrow[line width=.5mm]{>}}
    },
    ]
    \draw[geodesic, postaction={decorate}] (p0) -- (q0)
    node[midway, label={left:$v_{i}$}] (v0) {};
    \draw[geodesic, postaction={decorate}] (q1) -- (p1)
    node[midway, label={right:$v_{i+1}$}] (v1) {};

  \end{scope}

\end{tikzpicture}
\caption{
  An illustration of the control points and the corresponding velocities.
  We define a velocity $v_i$ at every point $p_i$.
  From that velocity, we construct the control point $q_{i}^+ \coloneqq E_{p_i}(v_{i}) \act p_i$.
  Similarly, we use $v_{i+1}$ to construct the point $q_{i+1}^{-} \coloneqq E_{p_{i+1}}(-v_{i+1}) \act p_{i+1}$.
  We can then construct a cubic Bézier curve from the four points $p_i$, $q_{i}^+$, $q_{i+1}^-$, $p_{i+1}$ as indicated in \autoref{fig:decasteljau}.
  If the velocities $v_i$ are chosen following \autoref{thm:main_result_C2}, the resulting piecewise cubic curve will have optimal regularity.
}
\label{fig:controlpoints}
\end{figure}

We shall now construct a fixed point iteration algorithm for cubic splines on symmetric spaces, based on the $C^2$ condition in \autoref{thm:main_result_C2}.
To this end, we parameterize the control points $q_i^-$ and $q_i^+$ by the corresponding tangent vectors $v_i$ as defined in \autoref{thm:main_result_C2}.
Next, let
\begin{equation}
  \begin{split}
    &\wel_i(\vel_{i-1},\vel_i,\vel_{i+1}) \coloneqq \\
    &\qquad \log_{p_i}\paren[\big]{E_{p_i}(-\vel_i) E_{p_{i+1}}(-\vel_{i+1})\cdot p_{i+1}} \\
    &\qquad -\log_{p_i}\paren[\big]{E_{p_i}(\vel_i) E_{p_{i-1}}(\vel_{i-1})\cdot p_{i-1}}
    .
\end{split}
\label{eq:weldef}
\end{equation}
The iteration map over $\bvec{\vel} \coloneqq (\vel_1,\dotsc,\vel_{N-1})$ is then given by
\begin{equation}
  \bvec{\vel} \mapsto \frac{1}{4}{\bvec{\wel}\paren{\bvec{\vel}}} + \frac{1}{2}\bvec{\vel}
  ,
\label{eq: fixedpointiter}
\end{equation}
where
\begin{equation}
\bvec{\wel}(\bvec{\vel}) \coloneqq \paren[\big]{\wel_1(\vel_0,\vel_1,\vel_2),\dotsc,\wel_{N-1}(\vel_{N-2},\vel_{N-1},\vel_N)}
.
\end{equation}
Pseudo code for the complete algorithm is given in \autoref{alg:mainalgorithm}.
We consider two options for boundary conditions.
\begin{enumerate}[label=\upshape(\roman*)]
\item Clamped spline:
\begin{equation}
\vel_0(\vel)= \vel_{\text{start}}, \qquad \vel_N(\vel)= \vel_{\text{end}}
\label{eq:bndclamp}
\end{equation}
 where $\vel_{\text{start}}$ and $\vel_{\text{end}}$ are prescribed constants.
\item Natural spline:
\begin{equation}
\vel_0(\vel) = \frac{1}{2} \log_{p_0}(q_1^-), \qquad \vel_N(\vel) = -\frac{1}{2}\log_{p_N}(q_{N-1}^+)
.
\label{eq:bndnat}
\end{equation}
\end{enumerate}

\begin{algorithm}
  \caption{Computing the control points $q_i^+$ and $q_i^-$}
  \label{alg:mainalgorithm}
The auxiliary functions $E$ and $\log$ are defined in \eqref{eqdef:movement} and \eqref{eqdef:log}.
  \begin{algorithmic}
    \State $\bvec{\vel} \gets (0,\dotsc,0)$ 
    \Repeat
    \State $\bar{\bvec{\vel}} \gets (\vel_0(\vel),\bvec{\vel},\vel_N(\vel))$ \Comment{$\vel_0$ and $\vel_N$ are determined by \eqref{eq:bndclamp} or \eqref{eq:bndnat}}
    \For{$i \gets 0,\dotsc,N-1$}
    \State $g_i^+ \gets E_{p_i}(\bar{\vel}_i)$
    \State $q_i^+ \gets g_i^+ \act p_i$
    \EndFor
    \For{$i \gets 1,\dotsc,N$}
    \State $g_i^- \gets E_{p_i}(-\bar{\vel}_i)$
    \State $q_{i}^- \gets g_i^- \act p_{i}$
    \EndFor
    \For {$i \gets 1, \dotsc, N-1$}
    \State $\delta_i \gets {\log_{p_i} (g_i^- \act q_{i+1}^-) -  \log_{p_i} (g_i^+ \act q_{i-1}^+)} - {2}\vel_i$
    \State $\vel_i \gets \vel_i + \frac{1}{4}\delta_i$
    \EndFor
    \Until{$\norm{\delta} \leq TOL$}
  \end{algorithmic}
    The outcome is the control points $q_i^+$ and $q_i^-$, from which one can compute the spline segments.
\end{algorithm}


\subsection{Convergence result}
\label{sec:convresult}

In this section we give a result on convergence of the fixed point method given by \autoref{alg:mainalgorithm}.
To do so, we first give some preliminary definition.

The Riemannian distance between $p,q\in M$ is denoted by
$
  d(p,q)
$.
Our proof of convergence uses that two consecutive points $p_i,p_{i+1}$ are close.
Thus, we define the constant
\begin{equation}
\label{eq:maxdistance}
  D = \max_{0\leq i \leq N-1} d(p_i,p_{i+1}).
\end{equation}
The other constant that comes into play is a bound on (essentially) the curvature of $M$.
Indeed, if $\lVert\cdot\rVert_p$ denotes the $H_p$-invariant norm on $\lie{m}_p$ associated with the Riemannian structure of $M$, then we
define the constant $K\geq 0$ as
\begin{equation}
\label{eq:defcurvature}
  K^2 = \sup_{\xi,\eta,\zeta \in \lie{m}_p\backslash 0} \frac{\lVert [\xi,[\eta,\zeta]]\rVert}{\lVert\xi\rVert \lVert\eta\rVert \lVert\zeta\rVert}.
\end{equation}
If $M$ is flat, then $K=0$.
If $M$ is the $n$-sphere of radius $r$, then $K= 1/r$.

Our convergence result is now formulated as follows.

\begin{theorem}\label{thm:main_result_convergence}
If $KD$ is sufficiently small, then \autoref{alg:mainalgorithm} with natural boundary conditions converges linearly.
If $KD$ and $\max\set{ K\norm{v_0}, K\norm{v_N}}$ are both sufficiently small, then \autoref{alg:mainalgorithm} with clamped boundary conditions converges linearly.
\end{theorem}

\begin{proof}[Outline of proof]
  Again, the proof is long and technical, and therefore given in \autoref{sec: conv}.
  Here we give a brief outline.

  The proof is based on showing that the iteration mapping \eqref{eq: fixedpointiter} is a contraction.
  For convenience, we use the variables $\xi_i \in \lie{m}_{p_i}$ instead of $v_i \in \Tan_{p_i}M$.
  The iteration mapping is denoted $\phi$.
  \begin{enumerate}[label=\textup{(\roman*)}]
    \item 
     The first step is to give an invariant region of the iteration mapping.
    This is given in \autoref{prop: invarea}, which proves existence of $V>0$ (depending on both $D$ and $K$), such that $\lVert \bvec{\xi}\rVert \leq V$ implies $\lVert\phi(\bvec{\xi})\rVert \leq V$.
    \item 
    The next step is to prove that $\phi$ is a contraction, which is established by a series of estimates on the partial derivatives on $\phi$ (\autoref{prop: ondiag}, \autoref{prop: offdiag}, and \autoref{prop: boundary}).
    These estimates depend on $K$.
    \item 
    The final step, in \autoref{sub:convergence_proof_final}, is to use the contraction mapping theorem. 
  \end{enumerate}
\end{proof}



\section{Examples}\label{sec:examples}

\subsection{Unit quaternions}
Unit quaternions (versors) are used extensively in computer graphics to represent 3-D rotations.
They are elements of the form $q=q_0 + q_1 i + q_2j + q_3 k$ with
\begin{equation}
  q_0^2 + q_1^2 + q_2^2 + q_3^2 = 1.
\end{equation}
Thus, we can identify unit quaternions with $\Sphere{3}$.
In turn, $\Sphere{3}$ is a double cover of $\SO(3)$.
A point on $\Sphere{3}$ therefore gives a rotation matrix, and likewise any rotation matrix corresponds to two antipodal points on $\Sphere{3}$.
The rotation of a vector $p=(p_x,p_y,p_z)$ by a unit quaternion $q$ can be compactly written using quaternion multiplication as
\[\rota_q (p)= qpq^{-1}\]
where $p = p_x i + p_y j + p_z k$ is thought of as a pure imaginary quaternion.

Since $\Sphere{3} \simeq \OO(4)/\OO(3)$ is a Riemannian symmetric space (with respect to the standard Riemannian metric), we can use \autoref{alg:mainalgorithm} to obtain $C^2$~continuous spline interpolation between rotations.
Since geodesics on $\Sphere{3}$ are given by great circles, it is straightforward to derive the mappings $E$ and $\log$.

The resulting $C^2$-curve interpolating 5 orientations is shown in \autoref{fig:rotation}. 
In the figure, an element $\gamma(t) \in \Sphere{3}$ is represented by the rotated basis vectors \[\set{\rota_{\gamma(t)}(i),\rota_{\gamma(t)}(j),\rota_{\gamma(t)}(k)}.\]
The actual interpolation points are marked with bolder lines.

Note that $\Sphere{3}$ can be canonically identified with $\SU(2)$.
Note also that $\SU(2)$ has a symmetric space structure as the quotient $(\SU(2) \times \SU(2))/\SU(2)$ \cite{MKVe16}.
Here, however, we considered the symmetric space structure $\Sphere{3} = \SO(4)/\SO(3)$.
These two symmetric space structure are locally equivalent.
This is because, $\SU(2) \times \SU(2)$ is a double cover of $\SO(4)$ and $\Sphere{3}$ is a double cover of $\SO(3)$.
Our algorithm gives the same result regardless of the choice of either of those two symmetric space structures on $\Sphere{3}$.
From this point of view, our algorithm could be used for spline interpolation on any compact Lie group $G$, using the Cartan--Schouten symmetric space structure $(G\times G)/G$. 

\begin{figure}
  \centering
  \includegraphics[width=.8\textwidth]{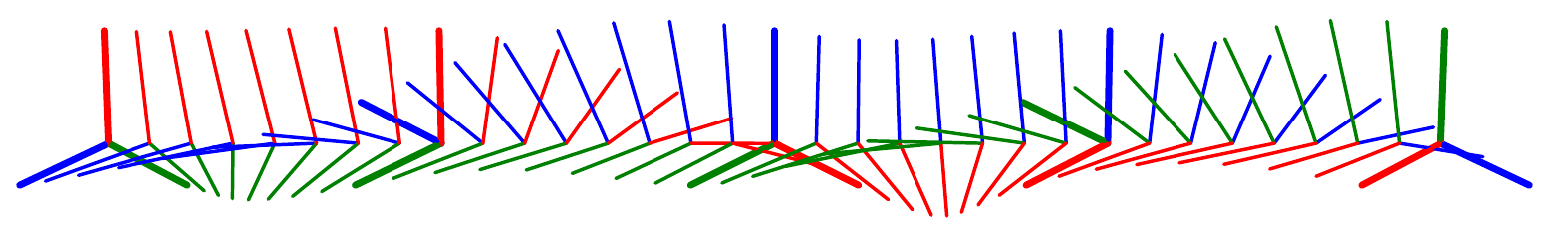}
  \caption{Interpolation between 5 orientations, illustrated by the orientation of the frame consisting of the axes $i$ (red), $j$ (green), and $k$ (blue). 
  The bold frames mark the interpolants.}
  \label{fig:rotation}
\end{figure}

\subsection{Quantum states}
\label{sec:quantumexample}
The control of quantum states is an important subproblem in quantum information and the realization of quantum computers~\cite{DoPe2010}.
The objective is to find a time-dependent Hamiltonian $H(t)$ designed to `steer' a given quantum state $|\psi_0\rangle$ through a sequence of states $|\psi_1\rangle,|\psi_2\rangle,\ldots,|\psi_N\rangle$ at given times $t_1,\ldots,t_N$.
As instantaneous switching of the Hamiltonian is not experimentally feasible \cite{BrHoMe2012}, the interpolating curve $|\psi(t)\rangle$ should be at least $C^2$ continuous.

If the quantum state space is finite dimensional, corresponding to an ensemble of qubits, then, in the geometric description of quantum mechanics \cite{Ki1978,Ki1979}, the phase space is given by complex projective space $\CC P^n$.
The quantum control problem can then be seen as a two-step process:
\begin{enumerate}
  \item Find an interpolating curve $t\mapsto |\psi(t)\rangle \in \CC P^n$ such that $|\psi(0)\rangle = |\psi_0\rangle$ and $|\psi(t_k)\rangle = |\psi_k\rangle$.
  \item Using the homogeneous structure $\CC P^n \simeq \UU(n+1)/\UU(n)\times \UU(1)$, lift $|\psi(t)\rangle$ to a curve $t\mapsto g(t) \in \UU(n+1)$ such that $g(0) = e$ and $\pi(g(t)) = |\psi(t)\rangle$.\footnote{The lifting is, of course, not unique.
  It is natural to minimize the change in Hamiltonian $\dot H$, as suggested by \citet*{BrHoMe2012}.}
  The time-dependent Hamiltonian is then given by $H(t) = - i \dot g(t) g(t)^{-1}$.
\end{enumerate}
We can use \autoref{alg:mainalgorithm} for the first step.
(The second step is not treated in this paper.)

The simplest case of a single qubit corresponds to the phase space $\CP{1}$, which is isomorphic to a sphere (called the \emph{Bloch sphere}).
Using 6 interpolation points, the resulting interpolating curve is visualized on the Bloch sphere in \autoref{fig:quantum}.

\begin{figure}
  \centering
  \includegraphics[width=.6\textwidth]{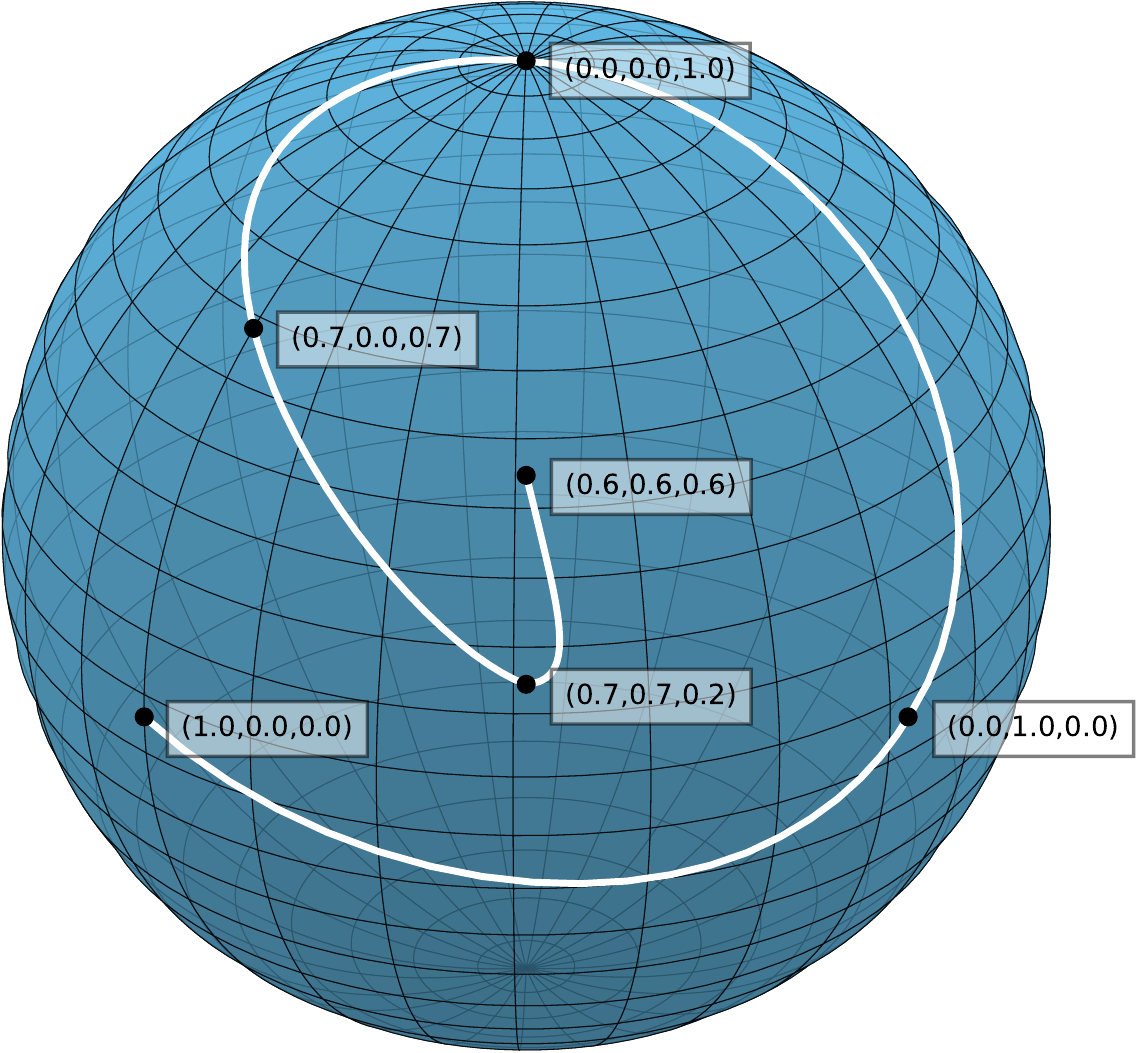}
  \caption{Interpolation between 6 qubit states visualized on the Bloch sphere.
  The resulting curve is $C^2$ and therefore generated by a time-dependent Hamiltonian $iH(t) \in \mathfrak{u}(n+1)$ where $t\mapsto\dot H(t)$ is continuous.
  }
  \label{fig:quantum}
\end{figure}

\subsection{Affine shapes}
\label{sec:visionexample}

The space of \emph{affine shapes} is the quotient set whose elements consist of $k+1$ labelled (\emph{i.e.}~ordered) points in $\RR^n$ defined up to affine transformations.
Thus, $(p_0,\ldots,p_k)\in (\RR^n)^{k+1}$ and $(\tilde p_0,\ldots,\tilde p_k) \in (\RR^n)^{k+1}$ define the same affine shape if there exists an affine transformation $p\mapsto Ap+q$ such that $(\tilde p_0,\ldots,\tilde p_k) = (Ap_0 + q,\ldots,Ap_k + q)$.
The study of affine shapes has many applications in computer vision~\cite{KeBaCaLe09}.
An example is the identification of planar surfaces (such as facades of buildings) in images taken from different projection angles.

The space of affine shapes with \emph{full rank} (meaning that the $k+1$ points are not contained in an affine hyperplane of $\RR^n$) can be identified with the real Grassmannian manifold $\mathrm{Gr}(n,k)$ of $n$-dimensional subspaces in $\RR^k$.
Indeed, choose a coordinate system
and place the coordinates of the vectors $p_1 - p_0,\dotsc,p_k-p_0$ column by column in an $n \times k$ matrix.
The kernel of this matrix
is then a subspace of $\RR^k$.
The full rank assuption means
that the kernel has dimension $k-n$.
This gives us a point in $\mathrm{Gr}(k-n,k)$,
which,
by the canonical isomorphism between $\mathrm{Gr}(k-n,k)$ and $\mathrm{Gr}(n,k)$,
gives a point in $\mathrm{Gr}(n,k)$.

Since the Grassmannian $\mathrm{Gr}(n,k)$ is a symmetric space, we can use our algorithm to construct $C^2$-splines between affine shapes.

Let us study a particular case of affine shapes corresponding to $3+1$ points in $\RR^2$.
Our aim is to create a spline between such shapes, defined by (A)--(D) in \autoref{fig:affineshape}.
For every shape,
we obtain a $2\times 3$ matrix of full rank, which defines our element in $\mathrm{Gr}(2,3)$.
We then use \autoref{alg:mainalgorithm} to construct a spline interpolating the points in $\mathrm{Gr}(2,3)$.

To visualize the result, notice that $\mathrm{Gr}(2,3)$ is isomorphic to the real projective plane $\RR P^2$. 
Indeed, this follows since the kernels of the $2\times 3$ matrices are one dimensional,
that is, they are directions in $\RR^3$.
We may therefore visualise the affine shapes as follows.
First, we compute the intersection of the kernel direction with the half-sphere
\begin{equation}
  \setc[\big]{(x,y,z) \in \RR^2}{z \leq 0}
  .
\end{equation}
We then compute the stereographic projection from the point of coordinate $(0,0,1)$.
This function is explicitly given by 
\begin{equation}
  (x,y,z) \mapsto \paren[\Big]{\frac{x}{1-z}, \frac{y}{1-z}}
  .
\end{equation}
The resulting projected spline curve is plotted in \autoref{fig:stereographic}.

We stress that interpolation of affine shapes cannot, in general, be accomplished by standard spline interpolation in vector spaces.
Indeed, if we introduce local coordinates by expressing the point $p_2 = a e_0 + b e_1$ in the basis spanned by the basis vectors $e_0 = p_1-p_0$ and $e_1 = p_3-p_0$ with origin at $p_0$, then the $b$-coordinate is infinite as the spline passes through the affine shape (C).
This is illustrated in \autoref{fig:shapecoordinate}.
More generally, any local coordinate chart would have similar artifacts for some shapes.

\tikzset{caxis/.style={thick, ->}}
\tikzset{xaxis/.style={}}
\tikzset{yaxis/.style={}}

\newcommand*\drawpoints[7]{
  \begin{tikzpicture}[scale=1.2]
    \node (p0) at (0,0) {};
    \node (pl) at (-1,0) {};
    \node (pr) at (1,0) {};
    \node (q) at (0,1) {};
    \filldraw[gray]
    (p0) circle [radius=2pt]
    (pr) circle [radius=2pt]
    (q) circle [radius=2pt]
    (pl) circle [radius=2pt]
    ;
    \path (q) ++(0,-0.25) node {#1};
    \path (pl) ++(0,-0.25) node {#2};
    \path (p0) ++(0,-0.25) node {#3};
    \path (pr) ++(0,-0.25) node {#4};
    \path[xaxis] (#5) -- (#6);
    \path[yaxis] (#5) -- (#7);
    \end{tikzpicture}
}

\begin{figure}
  \centering
  \subfloat[{}]{
      \drawpoints{0}{1}{2}{3}{q}{pl}{pr}
}\qquad
\subfloat[]{
    \drawpoints{1}{0}{2}{3}{pl}{q}{pr}
}\\
\subfloat[]{ 
    \drawpoints{2}{0}{1}{3}{pl}{p0}{pr}
}\qquad
\subfloat[]{ 
    \drawpoints{3}{0}{1}{2}{pl}{p0}{q}
}
  \caption{
  (A)--(D) shows different sets of labelled points $p_0,p_1,p_2,p_3$ in $\RR^2$.
  Each such set defines an affine shape, corresponding to an element in $\mathrm{Gr}(3,2)$.
  Using \autoref{alg:mainalgorithm} we construct a spline interpolating between the affine shapes (A)--(B)--(C)--(D)--(A), thus giving a closed curve.
  The resulting spline is visualized in \autoref{fig:stereographic}.
  }
\label{fig:affineshape}
\end{figure}


\begin{figure}
  \centering
  \includegraphics[width=.45\textwidth]{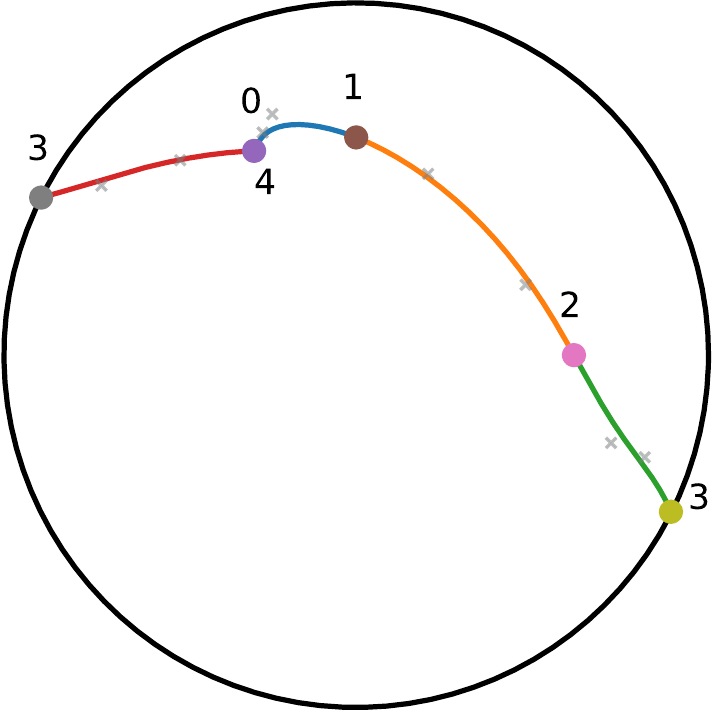}.
  \caption{
    Using that $\mathrm{Gr}(3,2)\simeq \RR P^2$, we visualize the spline interpolating the affine shapes (A)--(B)--(C)--(D)--(A) in \autoref{fig:affineshape} using a stereographic projection.
    As we identify any two opposite points of the circle,
    the curve is closed.
  }
\label{fig:stereographic}
\end{figure}

\begin{figure}
  \centering
  \includegraphics[width=.8\textwidth]{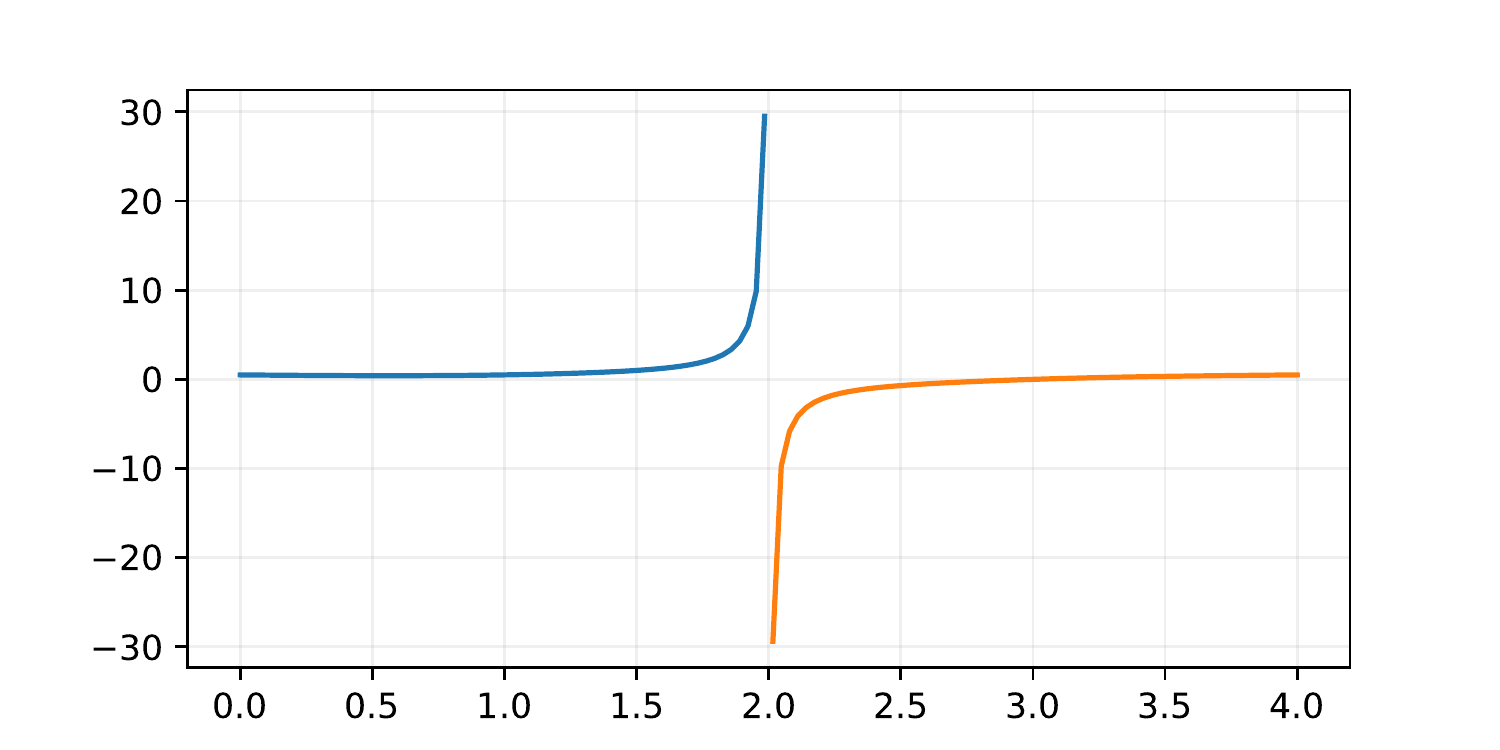}
  \caption{This figure illustrates why it is not possible, in general, to use a local coordinate chart to construct interpolation between affine shapes.
  In this example, the affine shape spline is described in a coordinate system obtained by placing the origin at $p_0$, with basis vectors $e_0 = p_1-p_0$ and $e_1 = p_3-p_0$, and then expressing $p_2 = ae_0 + b e_1$ in this basis. 
  Then $(a,b)$ constitutes local coordinates for $\mathrm{Gr}(3,2)$.
  However, there is a singularity at the shape (C), which implies that the $b$--coordinate becomes infinite there.
  This effect is only apparent in these local coordinates: nothing is infinite in the actual computed spline curve,
  as shown in \autoref{fig:stereographic}.
  Any other choice of local coordinates suffers analogous problems.
  }
  \label{fig:shapecoordinate}
\end{figure}

\section{Outlook}\label{sec:outlook}
There are a number of possible extensions of our methods and proofs to be studied in future work.
Here we list some of them.
\begin{itemize}
\item
  We surmise that the conditions of~\autoref{thm:main_result_C2} are in fact valid for more general symmetric spaces than Riemannian ones.
  In particular, we would like to be able to prove that result directly, without resorting to the previous result of Popiel and Noakes, which is the only reason for us to restrict to the \emph{Riemannian} symmetric space case at this point.
\item
  The only structure that we need to implement~\autoref{alg:mainalgorithm} and to state~\autoref{thm:main_result_C2}
  is an invariant connection (see~\eqref{eq:connection}), which only requires the homogeneous space at hand to be \emph{reductive}~\cite[\S\,4]{MKVe16}.
  We would expect similar results in that case, which would for instance cover arbitrary Lie groups, as well as Stiefel manifolds.
\item
  A modification of our algorithm is possible in the more general case of \emph{arbitrary} Riemannian manifolds (not necessarily symmetric), using the original $C^2$ condition of \citet{PoNo2007}.
  We are working on an implementation and a proof of convergence.
\item
  Our algorithm is currently based on a fixed point iteration.
  One could instead use a Newton iteration, which would ensure convergence in one step in the Euclidean case.
\end{itemize}

\subsection*{Acknowledgments}
This project has received funding from the European Union’s Horizon 2020 research and innovation programme under grant agreement No 661482, from the Swedish Foundation for Strategic Research under grant agreement ICA12-0052, and from the Knut and Alice Wallenberg Foundation under grant agreement KAW-2014.0354.

\appendix

\section{Identities and bounds on symmetric spaces}
\label{sec: lemmata}

In this section we prove several identities and bounds on (Riemannian) symmetric spaces which are used in \autoref{sec:C2reg} and \autoref{sec: conv} below.

Let $M= G/H_p$ be a symmetric space,
where $p\in M$ and $\lie{g}=\lie{m}_p\oplus \lie{h}_p$ is the decomposition described in \autoref{sub:symmetric_spaces},
and let $\pi_p\from \lie{g} \to \lie{m}_p$ be the canonical projection.
Denote by $\Psi^p\from G \to M$ and $\tilde{L}_g\from M\to M$ the maps
\[
  \Psi^p(g)=\tilde{L}_g(p)=g\cdot p
  ,
\]
as usual, $L_g\from G\to G$ and $R_h\from G\to G$ are defined by
\[L_g(h)=R_h(g)=gh \]
and we identify $\lie{g}=\Tan_eG$ to define
\[\Tan_eL_{g}\from \lie{g}\to \Tan_g G.\]

By an abuse of notation, we will use $g\act$ to denote any of the following action maps
\[
  \begin{aligned}
    \tilde{L}_g &\from M \to M,\\
    \Tan\tilde{L}_g &\from \Tan M \to\Tan M\\
\Ad_g \coloneqq\Tan_{g} R_{g^{-1}}\Tan_{e} L_{g}  &\from \lie{g}\to \lie{g} .
    \end{aligned}
  \]

The notation $\pairing{\cdot}{\cdot}_p$ is used to denote the natural pairing of
one-forms and tangent vectors over a point $p$.

The infinitesimal action \eqref{eq:inf_action} and the decomposition $\lie{g}=\lie{m}_p\oplus \lie{h}_p$  define a \emph{principal connection}~\cite{MKVe16}, i.e.,
a $\lie{g}$-valued one-form $\varpi$ on $M$
which fulfills
\begin{subequations}
\begin{align}
\label{eq:connection}
  \pairing{\varpi}{v}_p \act p &= v & \text{(consistency)}\\
  \pairing{\varpi}{g \act v}_{g\act p} &= g \act \pairing{\varpi}{v}_p
                                           & \text{(equivariance)}
\end{align}
\end{subequations}
for all $v \in \Tan_pM$.

By construction, $\lie{m}_{p} = \varpi(T_pM)$.

Throughout this section and the following, $\dexp$ denotes the trivialized derivative of the Lie group exponential, that is
\[\Tan_\xi \exp  = T_eR_{\exp(\xi)} \dexp_\xi  = T_eL_{\exp(\xi)} \dexp_{-\xi}
,
\]
where the  last equality follows from differentiating $\exp(\xi)\exp(-\xi) =
\id$.

\begin{lemma}
\label{prop:expident}
Fix $p \in M$, and define $q\from \lie{g}\to M$ by
$q(\xi)=\exp(\xi)\act p$.
Then
\[\Tan_{\xi}q (\delta \xi)= \exp(\xi)\act (\pi_p \dexp_{-\xi}(\delta \xi) \act p), \qquad\forall\, \delta\xi\in\lie{g}.
\]
\end{lemma}
\begin{proof}
By the properties of a Lie group action, for $\eta \in \lie{g}$, we have
\[\Tan_g\Psi^p(g\act \eta)= g\act(\eta\act p)\]

We differentiate $q=\Psi^p\circ \exp$ by the product rule and get
\[\begin{aligned}
      \Tan_{\xi} q(\delta \xi) &= \Tan_{\exp(\xi)}\Psi^p  \circ \Tan_{\xi}\exp (\delta \xi)\\
      &= \exp(\xi) \act \left(  \dexp_{-\xi}(\delta \xi) \act p\right)\\
      &= \exp(\xi) \act (\pi_p \dexp_{-\xi}(\delta \xi) \act p)
    \end{aligned}
  \]
    where the final equality is because $\pi_p(\eta)\act p = \eta\act p$ for all $\eta \in \lie{g}$.
\end{proof}

\begin{lemma}
\label{prop:adeven}
Fix $p \in M$.
  Let $\xi, \eta \in\lie{m}_p$, then
\[\pi_p\dexp_{-\xi}(\eta)= \frac{\sinh(\ad_\xi)}{\ad_\xi}(\eta)= \pi_p \dexp_{\xi}(\eta).
\]
\end{lemma}
\begin{proof}
  As a function $\lie{g}\to \lie{g}$, $\dexp_{\xi}$ is an analytic function of $\ad_{\xi}$ with Taylor series $\dexp_{\xi} = \sum_{k=0}^{\infty} \frac{1}{(k+1)!} \ad_\xi^k$.
  Using the definition \eqref{eq: algcond} of a symmetric space, we have, for any element $\eta\in \lie{m}_p$, that
\[
\ad_{\xi}^k\eta \in \begin{cases} \lie{m}_p, & k \text{ even}\\
                                    \lie{h}_p, & k \text{ odd}
                                    \end{cases}
\]
The effect of $\pi_p$ is thus to eliminate terms in $\lie{h}_p$, leaving only the terms with an even exponent $k$, so
\[
  \pi_p  \dexp_{\xi} \eta =  \sum_{l=0}^\infty \frac{1}{(2l+1)!} \ad_{\xi}^{2l}\eta  = \frac{\sinh(\ad_\xi)}{\ad_\xi}\eta
  .
\]
Note that, as $ {\sinh(\ad_\xi)}/{\ad_\xi}$ is an even function of $\xi$, we have
\[
\pi_p \dexp_{\xi} \eta
=
\pi_p \dexp_{-\xi} \eta
.
\]

\end{proof}
A corollary of \autoref{prop:adeven} is that for $\xi$ near zero,
$\pi_p\dexp_{\xi}$ is invertible as a function $\lie{m}_p\to \lie{m}_p$ with
inverse given by $\ad_{\xi}/\sinh{\ad_{\xi}}$.

\begin{lemma}
\label{prop:involution}
Fix $p \in M$.
Let $I_p$ be the symmetry function at $p \in M$, i.e., $I_p(q) = E_p(-\log_{p} (q)),$ and let $\xi \in \lie{m}_p$, $\delta p \in \Tan_p M$.
Then
\[
\Tan_{\exp(\xi)\cdot p} I_p ( \exp(\xi)\act \delta p) = - \exp(-\xi)\act \delta p .
\]
\end{lemma}
\begin{proof}
  Let $q(\xi)= \exp(\xi)\act p$. Then $q\from \lie{m}_p \to M$ is a local diffeomorphism and
$I_p(q(\xi)) = q(-\xi)$.
Differentiating both $q(\xi)$ and $I_p(q(\xi))$ with respect to $\xi$, and using \autoref{prop:expident}, we get
\[\begin{aligned}
\Tan_\xi q (\eta)  = \exp(\xi) \act \left( \pi_p \dexp_{-\xi}(\eta) \act p\right)\\
\Tan_{q(\xi)} I_{p_i}\circ \Tan_\xi q  (\eta)   = -\exp(-\xi) \cdot \left(\pi_p \dexp_{\xi}(\eta)\act p\right).
\end{aligned}
\]
Since $q$ is a diffeomorphism, we can choose $\eta\in \lie{g}$ such that $\Tan_{\xi} q(\eta) = \exp(\xi)\act w$, \emph{i.e.}, $w =  \pi_p \dexp_{-\xi}(\eta)\act p$.
Using \autoref{prop:adeven}, we also have $w = \pi_p \dexp_{\xi}(\eta)\act p$, and the result follows.
\end{proof}

\begin{lemma}
\label{prop:logequi}
  Let $g\in G$, and $p,q\in M$. Then
\[
 \log_{g\act p }(q)= g \act \log_p(g^{-1}\act q).
\]
\end{lemma}

\begin{proof}
Let $\sigma(t)=[g\act p, q]_t$ be the interpolating curve between $g\act p$ and
$q$.
Then
$g^{-1}\act \sigma(t)=[p, g^{-1}\act q]_t$,
and differentiating at $t=0$, we get
\[
g^{-1}\act \log_{g\act p}(q) = \log_p(g^{-1}\act q).
\]
\end{proof}

For the final lemmata, we assume $M = G/H_p$ to be a Riemannian symmetric space.
This is equivalent to the existence of an $H_p$-invariant inner product on $\lie{m}_p$, from which we derive a norm $\norm{\cdot}$ on $\lie{m}_p$.

\begin{lemma}
\label{lem: dexpbound}
Let $\xi, \eta \in \lie{m}_p$, then
\[
\begin{aligned}
\norm{\pi_p  \dexp_{\xi} (\eta)} = \norm*{ \frac{\sinh(\ad_\xi)}{\ad_\xi} (\eta)} &\le
\frac{\sinh(K\norm{\xi})}{K\norm{\xi}}\norm{\eta},\\
\norm{(\pi_p \dexp_{\xi})^{-1}(\eta)} = \norm*{ \frac{\ad_\xi}{\sinh(\ad_\xi)}(\eta)} &\le \frac{K\norm{\xi}}{\sin(K\norm{\xi})}\norm{\eta}.
\end{aligned}
\]
(Notice $\sin$, not $\sinh$, in the last denominator.)
\end{lemma}
\begin{proof}
We use \autoref{prop:adeven} to set $\pi_p \dexp_{\xi}= \frac{\sinh(\ad_\xi)}{\ad_\xi}$.
The lemma at hand then follows from considering the infinite series for
$\frac{\sinh(\ad_\xi)}{\ad_\xi}$ and its inverse, as well as the definition of the curvature $K$ in equation \eqref{eq:defcurvature}.
\end{proof}

\begin{lemma}
\label{lem: distbound}
Fix $p\in M$.
Let $\xi_0, \xi_1 \in \lie{m}_p$, $q_0= \exp(\xi_0)\act p$ and $q_1 = \exp(\xi_1)\act p$.
Recall from \autoref{sec:convresult} that $d$ is the Riemannian distance between two points on $M$.
Let us define
\[
U \coloneqq \frac{1}{2} \left(\norm{\xi_0}+\norm{\xi_1}+ d(q_0,q_1) \right)
.
\]
We then have the inequality
\[
\norm{\xi_0-\xi_1} \le \frac{KU}{\sin(KU)} d(q_0,q_1).
\]
\end{lemma}
\begin{proof}
Let $\gamma(t)$ be the geodesic from $\gamma(0) = q_0$ to $\gamma(1) = q_1$, and let
$\alpha(t) \in \lie{m}_p$ be such that
\begin{equation}
\exp(\alpha(t))\act p = \gamma(t)
\label{eq: alphaeq}
.
\end{equation}
Then $\xi_0= \alpha(0)$, $\xi_1 = \alpha(1)$.
We note that, by the triangle inequality, we have
\begin{equation}\begin{aligned}
     \norm{\alpha(t)} = d(p, \gamma(t)) &\le \min \cbrace{ d(p, q_0)+d(q_0, \gamma(t)), d(p, q_1)+d(q_1, \gamma(t))}\\
    &\le\min \cbrace[\Big]{ \norm{\xi_0} + t d(q_0,q_1) , \norm{\xi_1} + (1-t)d(q_0, q_1)}\\
                                    &\le \frac{1}{2}\paren[\Big]{\norm{\xi_0}+\norm{\xi_1}+ d(q_0,q_1) } =U.
\end{aligned}
\label{eq: alphaineq}
\end{equation}
Differentiating \eqref{eq: alphaeq} we get
\[
\exp(\alpha(t)) \act (\dexp_{-\alpha(t)} \dot{\alpha}(t) \act p) = \dot{\gamma}(t)
,
\]
which we rewrite as
\[
(\pi_p  \dexp_{-\alpha(t)} \dot{\alpha}(t)) \act p = \exp(-\alpha(t))\act \dot{\gamma}(t),
\]
where we have used that $\pi_p(\eta) \act p = \eta\act p$ for all $\eta \in \lie{g}$.
Taking norms on both sides and using that $G$ acts by isometries, we have
\[
\norm{\pi_p  \dexp_{-\alpha(t)} \dot{\alpha}(t)} =  \norm{\exp(-\alpha(t))\act
  \dot{\gamma}(t)} = \norm{\dot{\gamma}(t)} = d(q_0, q_1).
\]
Moreover, by \autoref{lem: dexpbound},
\[
  \norm{\dot{\alpha}(t)} \le \frac{K\norm{\alpha(t)}}{\sin(K\norm{\alpha(t)})} d(q_0, q_1)
  .
\]
We now use the monotonicity of $x/\sin x$, and the bound \eqref{eq: alphaineq} to obtain
\[\begin{aligned}
\norm{\alpha(1)- \alpha(0)} \le& \int_0^1 \norm{\dot{\alpha}(t)} \ud t \\
\le& \int_0^1 \frac{K\norm{\alpha(t)}}{\sin(K\norm{\alpha(t)})}d(q_0, q_1) \ud t  \\
\le & \frac{KU}{\sin(KU)}d(q_0, q_1)
.
\end{aligned}
\]
\end{proof}

\section{\texorpdfstring{$C^2$}{C2} continuity}\label{sec:C2reg}

We now prove \autoref{thm:main_result_C2}. The proof relies on a result by Popiel and Noakes \cite{PoNo2007}.

\begin{proof}[Proof of \autoref{thm:main_result_C2}]
At non-integer $t$, the spline is $C^{\infty}$.
We show that the spline is $C^2$ at $t=i\in \set{1, 2, \dotsc , N-1}$, i.e., at the interpolating points $p_i$.
By \cite[Thm.~3]{PoNo2007}, a sufficient and necessary condition for the spline to be $C^2$ at $t=i$ is
\begin{equation}
\Tan_{q_i^+} I_{p_i}(\log_{q_i^+} q_{i+1}^-)  = -\log_{q_i^-}(q_{i-1}^+)-2\log_{q_i^-}(p_i)
,
\label{eq: C2pf1}
\end{equation}
where $I_{p_i}$ is the symmetry function at $p_i$.
We proceed by removing the dependence of $\Tan I_{p_i}$ in \eqref{eq: C2pf1} by using the lemmata from \autoref{sec: lemmata}. 

Recall that $q_i^+=\exp(\xi_i)\cdot p_i$ and $q_i^-=\exp(-\xi_i)\cdot p_i$.
By \autoref{prop:logequi}, we have
\[
\begin{aligned}
\log_{q_i^+} (q_{i+1}^-) &= \log_{\exp(\xi_i) \act p_i} (q_{i+1}^-)\\
                        &= \exp(\xi_i)\act\log_{p_i}(\exp(-\xi_i) \act q_{i+1}^-),\\
\log_{q_i^-} (q_{i-1}^+) &= \log_{\exp(-\xi_i) \act p_i} (q_{i-1}^+)\\
                        &= \exp(-\xi_i)\act\log_{p_i}(\exp(\xi_i) \act q_{i-1}^+)
                        .
\end{aligned}
\]
Using \autoref{prop:involution},  with $w= \log_{p_i}(\exp(-\xi_i) \act q_{i+1}^-)$, we can write
\begin{equation}\begin{aligned}
\Tan_{q_i^+} I_{p_i}(\log_{q_i^+} q_{i+1}^-) &= \Tan_{\exp(\xi_i)\act p_i } I_{p_i}(\exp(\xi_i)\act  \log_{p_i}(\exp(-\xi_i) \act q_{i+1}^-)\\
&=- \exp(-\xi_i)\act \log_{p_i}(\exp(-\xi_i) \act q_{i+1}^-)
.
\end{aligned}
\label{eq: tti1}
\end{equation}
Next, we note that
\begin{equation}\log_{q_i^-}(p_i) = \exp(-\xi_i)\act \vel_i.
\label{eq: tti2}
\end{equation}
Using \eqref{eq: tti1} and \eqref{eq: tti2}, we can rewrite \eqref{eq: C2pf1} as
\begin{equation*}
\begin{split}
  -&\exp(-\xi_i)\act \log_{p_i}(\exp(-\xi_i) \act q_{i+1}^-) =\\
  &\quad -\exp(-\xi_i)\act\log_{p_i}(\exp(\xi_i) \act \exp(\xi_{i-1}) \act p_{i-1})
    - 2 \exp(-\xi_i)\act \vel_i.
\end{split}
\end{equation*}
Acting on the whole equation from the left with $\exp(\xi_i)$ and rearranging, we get
\[\log_{p_i}(\exp(-\xi_i)\act q_{i+1}^-) - \vel_i =  \log_{p_i}(\exp(\xi_i)\act q_{i-1}^+) +\vel_i.
\]
We now obtain the $C^2$ condition \eqref{eq:c2_condition} by using the $C^1$ condition \eqref{eq:c1_condition} to replace $\vel_i = \log_{p_i} (q_i^+) = - \log_{p_i} (q_i^-)$ in the equation above.
\end{proof}

\section{Convergence of the fixed-point method}
\label{sec: conv}

We now proceed to prove \autoref{thm:main_result_convergence}, i.e., that if consecutive interpolation points are sufficiently close then \autoref{alg:mainalgorithm} converges to a solution.
In the proof, we will work with variables in the Lie algebra $\lie{g}$.
Let
\[
  \lie{m}_{\bvec{p}} \coloneqq \lie{m}_{p_0} \times \dotsm \times \lie{m}_{p_N}
  ,
  \]
and let
\[
  \bvec{\xi} \coloneqq (\xi_0, \dotsc, \xi_N) \in \lie{m}_{\bvec{p}},
\]
be the Lie algebra elements defined by
$\xi_i = \pairing{ \varpi}{\vel_i}_{p_i}.$

Switching to variables in the Lie algebra, the fixed-point iteration \eqref{eq: fixedpointiter} becomes a map
\[\begin{aligned}
\bvec{\phi}\colon \lie{m}_{\mathbf{p}}&\to \lie{m}_{\mathbf{p}},\\
\bvec{\phi}(\bvec{\xi}) &= (\xi_0, \dotsc, \xi_N) \to (\phi_0(\bvec{\xi}), \dotsc, \phi_N(\bvec{\xi})).
\end{aligned}
\]
For $1\le i\le N-1$, $\phi_i(\bvec{\xi}) =\frac{1}{4} \pairing{\varpi}{\wel_i(\bvec{\xi})}_{p_i} + \frac{1}{2}
\xi_i$.
To simplify notation, we define $L_p \from M \to \lie{m}_p$  by
\[L_p(q) \coloneqq \pairing{\varpi}{\log_p(q)}_p
  .\]
Now, for $0 < i < N$ the fixed-point iteration map $\phi_i$ can be expressed as
\[
  \begin{split}
    &\phi_i(\bvec{\xi}) =  \frac{1}{2} \xi_i + \\
     &\quad +\frac{1}{4}\paren[\Big]{L_{p_i} (\exp(-\xi_i)  \exp(-\xi_{i+1}) \act p_{i+1})- L_{p_i}(\exp(+\xi_i) \exp(+\xi_{i-1}) \act p_{i-1})}
    ,
\end{split}
\]
and both $\phi_0$ and $\phi_N$ are given by the boundary conditions
\[ \phi_0(\bvec{\xi}) = \pairing{\varpi}{\vel_0(\bvec{\xi})}, \qquad \phi_N(\bvec{\xi}) = \pairing{\varpi}{\vel_N(\bvec{\xi})}
  .
\]
Specifically, for clamped boundary conditions, we get
\begin{equation}
  \phi_0(\bvec{\xi}) = \xi_{\text{start}}, \qquad \phi_N(\bvec{\xi})= \xi_{\text{end}}
  \label{eq:clbc}
\end{equation}
where $\xi_{\text{start}}$ and $\xi_{\text{end}}$ are prescribed.
For free boundary conditions
\begin{equation}
\begin{aligned}
\phi_{0}(\bvec{\xi}) &= \frac{1}{2}\pairing{\varpi}{\log_{p_0}(\exp(-\xi_1)\act
  p_1)}_{p_0}, \\
  \phi_{N}(\bvec{\xi}) &= -\frac{1}{2}\pairing{\varpi}{\log_{p_{N}}(\exp(\xi_{N-1})\act
    p_{N-1})}_{p_N}.
\end{aligned}
\label{eq:frbc}
\end{equation}

It is convenient to introduce the auxiliary variables
\begin{equation}\begin{aligned}
\omega_i^+(\bvec{\xi})&= \pairing{\varpi}{\wel_i^+}_{p_i} =L_{p_i}\left(\exp(-\xi_i)  \exp(-\xi_{i+1}) \act p_{i+1}\right),\\
\omega_i^-(\bvec{\xi})&= \pairing{\varpi}{\wel_i^-}_{p_i}  = L_{p_i}\left(\exp(+\xi_i) \exp(+\xi_{i-1})\act p_{i-1}\right),
\end{aligned}
\label{eq: omega}
\end{equation}
and to write $\phi_i(\bvec{\xi}) = \phi_i^+(\bvec{\xi}) + \phi_i^-(\bvec{\xi}),$
where
\begin{equation}
\begin{aligned}
 \phi_i^+(\bvec{\xi})&= \frac{1}{4} \omega_i^+(\bvec{\xi})+\frac{1}{4} \vel_i,\\
 \phi_i^-(\bvec{\xi})&= -\frac{1}{4} \omega_i^-(\bvec{\xi})+\frac{1}{4} \vel_i.\\
 \end{aligned}
 \label{eq: omegapm}
 \end{equation}
Furthermore, we define the norms
\begin{equation}
\norm{\bvec{\xi}} \coloneqq \max_{0\le i \le N} \norm{\xi_i}= \max_{0\le i \le N} \norm{\vel_i}
\label{eq: norm}
\end{equation}
and
\[\norm{\bvec{\omega}(\bvec{\xi})} \coloneqq \max_{1\le i \le N-1} \max(\norm{\omega_i^+(\bvec{\xi})}, \norm{\omega_i^-(\bvec{\xi})}) .
\]

Recall that $D$  denotes the maximum distance between neighbouring interpolation
points as declared in equation \eqref{eq:maxdistance}.
Notice that by the triangle inequality, we have
\[ \norm{\omega_i^+(\bvec{\xi})} = d(q_i^+, q_{i+1}^-) \le d(p_i, p_{i+1}) + \norm{\xi_i} + \norm{\xi_{i+1}} \le D+ 2\norm{\bvec{\xi}}
\]
and similarly for $\norm{\omega_i^-(\bvec{\xi})}$,
so
\begin{equation}
\label{eq:omegabound}
\norm{\bvec{\omega}(\bvec{\xi})}\le D+  2\norm{\bvec{\xi}}
.
\end{equation}

We are now in position to prove \autoref{thm:main_result_convergence}.
The proof consists in showing the existence of an invariant region, and thereafter bounding the partial derivatives of $\bvec{\phi}$ to show that it is a contraction.

\subsection{Invariant region}

We begin by establishing an invariant region.

\begin{proposition}
\label{prop: invarea}
Let $V>0$ satisfy the inequality
\begin{equation}\frac{1}{2}\frac{K U}{\sin(KU)}(U-V) \le V
,
\label{eq: Vineq}
\end{equation}
where $U= D + 2V$.
If either one of the following boundary conditions is fulfilled:
\begin{enumerate}[label=\upshape(\roman*)]
\item clamped spline boundary conditions are used, $\norm{\xi_{\text{start}}} \le V$, $\norm{\xi_{\text{end}}} \le V$ and $\norm{\bvec{\xi}} \le V$,
\item natural spline boundary conditions are used and  $\norm{\bvec{\xi}} \le V$,
\end{enumerate}
then we also have $\norm{\bvec{\phi}(\bvec{\xi})} \le V$.
\end{proposition}

To reach this result, we first prove a few bounds for $\phi_i^+(\bvec{\xi})$.
\begin{lemma}
Let $1\le i\le N-1$, then the following bounds hold,
\[\begin{aligned}
\norm{\phi_i^+(\bvec{\xi})} &\le \frac{1}{4}\frac{KU_i}{\sin(KU_i)} (d(p_i, p_{i+1})+ \norm{\xi_{i+1}}),\\
\norm{\phi_i^-(\bvec{\xi})} &\le \frac{1}{4}\frac{KU_{i-1}}{\sin(KU_{i-1})} (d(p_i, p_{i-1})+ \norm{\xi_{i-1}}),
\end{aligned}
\]
where $U_i = d(p_i, p_{i+1})+ \norm{\xi_{i+1}}+ \norm{\xi_i}$.
\label{lem: phibound}
\end{lemma}
\begin{proof}
We prove the first inequality.
The proof of the second inequality is entirely symmetrical and therefore omitted.

By definition \eqref{eq: omegapm},
\[\phi_i^+(\bvec{\xi}) = \frac{1}{4}(\omega_i^+(\bvec{\xi})+ \xi_i)
  .
\]
The proof uses \autoref{lem: distbound} to bound $\omega_i^+(\bvec{\xi}) + \xi_i = \omega_i^+(\bvec{\xi})- (-\xi_i)$.

Recall that $G$ acts isometrically on $\lie{m}_p$ by the adjoint action
\[\begin{aligned}
\norm{\omega_i^+(\bvec{\xi})+ \xi_i} &= \norm{ \exp(\xi)\act  (\omega_i^+(\bvec{\xi}) + \xi_i)}\\
&=  \norm{ \exp(\xi_i)\act \omega_i^+(\bvec{\xi})+\xi_i },
\end{aligned}
\]
where $\act \omega_i^+(\bvec{\xi})$ and $\xi_i$ are vectors in
$\lie{m}_{q_i^+}$.
We also have
\[\exp(-\xi_i) \act q_i^+ = p_i\]
and
\[\begin{aligned}
\exp(\exp(\xi_i)\act \omega_i^+(\bvec{\xi}) )\act q_i^+ &= \exp(\xi_i)\exp(\omega_i^+(\bvec{\xi}))\exp(-\xi_i) \act q_i^+\\
&= \exp(\xi_i)\exp(\omega_i^+(\bvec{\xi}))\act p_i  \\
&= \exp(-\xi_{i+1}) \act p_{i+1}\\
&=q_{i+1}^+.
\end{aligned}
\]
By \autoref{lem: distbound}, we have
\[
\norm[\Big]{\exp(\xi_i)\cdot  \omega_i^+(\bvec{\xi})-(-\xi_i)} \le \frac{KT}{\sin(KT)} d(p_i, q_{i+1}^-)
\]
where
\[\begin{aligned}
2T & \coloneqq \norm{\exp(\xi_i) \omega_i} + \norm{\xi_i} + d(p_i, q_{i+1}^-)\\
  &= d(q_i^+, q_{i+1}^-) + \norm{\xi_i} + d(p_i, q_{i+1}^-).
\end{aligned}
\]
From the triangle inequality, we have
\[
\begin{aligned}
d(q_i^+, q_{i+1}^-) &\le d(p_i, p_{i+1}) + \norm{\xi_i} + \norm{\xi_{i+1}}, \\
d(p_i, q_{i+1}^-) &\le d(p_i, p_{i+1})  + \norm{\xi_{i+1}}
.
\end{aligned}
\]
Thus $T\le d(p_i, p_{i+1})+\norm{\xi_{i}}+\norm{\xi_{i+1}}= U_i$, and the claim follows by monotonicity of the function $\frac{x}{\sin x}$.
\end{proof}

\begin{proof}[Proof of {\autoref{prop: invarea}}]
Using \autoref{lem: phibound}, and the inequalities $\norm{\xi_i}\le \norm{\bvec{\xi}}\le V,$ $d(p_i, p_{i+1}) \le D$ and  $U_i \le U$, we get that

\[
\norm{\phi_i(\bvec{\xi})} \le \norm{\phi_i^-(\bvec{\xi})}+ \norm{\phi_i^-(\bvec{\xi})} \le \frac{1}{2}\frac{KU}{\sin(KU)} (D+V)
\]
for all $1\le i\le N-1$.

For the boundary velocities:
\begin{enumerate}[label=\upshape(\roman*)]
\item Clamped spline: $\phi_0(\bvec{\xi}) = \xi_{\text{start}}$, so $\norm{\phi_0(\xi)} \le V$ by our assumption.
\item Natural spline: $\phi_0(\bvec{\xi}) = \frac{1}{2}\pairing{\varpi}{\log_{p_0} (\exp(-\xi_1)\cdot p_1)}_{p_0}$, so
\[\norm{\phi_0(\bvec{\xi})} \le \frac{1}{2} (d(p_0,p_1)+ \norm{\xi_1}) \le \frac{1}{2} (D+V)\le V\]
under a weaker condition on $V$. Similarly for $\phi_N(\bvec{\xi})$.
\end{enumerate}
In conclusion, under the assumption \eqref{eq: Vineq}, we have
\[\norm{\bvec{\phi}(\bvec{\xi})} = \max_{0\le i \le N} \norm{\phi_i(\bvec{\xi})} \le V
\]
and $\setc[\big]{\bvec{\xi}}{\norm{\bvec{\xi}}\le V }$ forms an invariant region.
\end{proof}

\subsection{Contraction}

We now establish that, for $D$ and $\norm{\bvec{\xi}}$ sufficiently small, there
exists an $\alpha \in (0, 1)$ such that, in the operator norm derived from
\eqref{eq: norm}
$\norm{\Tan_{\bvec{\xi}} \bvec{\phi}} \le \alpha$.
We first consider the effect of varying $\xi_{i+1}$  and $\xi_{i-1}$ in $\phi_i(\bvec{\xi})$.
\begin{proposition}
  \label{prop: offdiag}
Let $1\le i \le N-1$ and denote by $\frac{\partial \phi_i(\bvec{\xi})}{\partial \xi_{j}}$
the partial derivative of $\phi_i$ with respect to $\xi_{j}$, i.e.
$\frac{\partial \phi_i(\bvec{\xi})}{\partial \xi_{j}}$ is a linear operator
$\lie{m}_{p_{j}}\to \lie{m}_{p_{i}}$. Then, in the operator norm, 
  \[\begin{aligned}
\norm*{\frac{\partial \phi_i(\bvec{\xi})}{\partial \xi_{i+1}}} &\le \frac{1}{4}
\norm[\big]{(\pi_{p_i}  \dexp_{-\omega_i^+(\bvec{\xi})})^{-1}}
\norm[\big]{\pi_{p_{i+1}} \dexp_{\xi_{i+1}}}\\
&\le \frac{1}{4} \frac{K\norm{\omega_i^+(\bvec{\xi})}}{\sin(K\norm{\omega_i^+(\bvec{\xi})})}\frac{\sinh(K\norm{\xi_{i+1}})}{K\norm{\xi_{i+1}}},
\end{aligned}\]
and
\[\begin{aligned}
\norm*{\frac{\partial \phi_i(\bvec{\xi})}{\partial \xi_{i-1}}} &\le \frac{1}{4} \norm{(\pi_{p_i} \dexp_{-\omega_i^-(\bvec{\xi})})^{-1}} \norm{\pi_{p_{i-1}} \dexp_{\xi_{i-1}}}\\
&\le \frac{1}{4} \frac{K\norm{\omega_i^-(\bvec{\xi})}}{\sin(K\norm{\omega_i^-(\bvec{\xi})})}\frac{\sinh(K\norm{\xi_{i-1}})}{K\norm{\xi_{i-1}}}
.
\end{aligned}\]
\end{proposition}
\begin{proof}
We only prove the first claim. The proof of the second is entirely symmetrical.
The term $\xi_{i+1}$ only enters $\phi_i(\bvec{\xi})$ through the term $\frac{1}{4}\omega_i^+(\bvec{\xi})$, so
\begin{equation}
  \frac{\partial \phi_i(\bvec{\xi})}{\partial \xi_{i+1}} = \frac{1}{4}
  \frac{\partial \omega_i^+(\bvec{\xi})}{\partial \xi_{i+1}}
  .
\label{eq: propoffdiag1}
\end{equation}

By the definition of $\omega_i^+(\bvec{\xi})$, we have
\[\exp(\omega_i^+(\bvec{\xi})) \act p_i = \exp(-\xi_i) \exp(-\xi_{i+1})\act p_{i+1}
  .
\]

Differentiating implicitly on both sides, then acting from the left by $\exp(-\omega_i^+(\bvec{\xi}))$, we get
\[
  \begin{split}
    &(\dexp_{-\omega_i^+(\bvec{\xi})}\delta \omega_i^+) \act p_i = \\
    &\qquad -\exp(-\omega_i^+(\bvec{\xi})) \exp(-\xi_i) \exp(-\xi_{i+1}) \act (\dexp_{\xi_{i+1}}\delta \xi_{i+1}) \act {p_{i+1}}
    ,
  \end{split}
\]
where $\delta \omega_i^+ = \frac{\partial \omega_i^+(\bvec{\xi})}{\partial
  \xi_{i+1}} \delta \xi_{i+1}$.
Taking norms on both sides, and using that the group action is isometric, we get
\[ \norm{(\dexp_{-\omega_i^+(\bvec{\xi})}\delta \omega_i^+) \act p_i} = \norm{(\dexp_{\xi_{i+1}}\delta \xi_{i+1}) \act {p_{i+1}}},\]
and
\[
  \norm{\pi_{p_i} \dexp_{-\omega_i^+(\bvec{\xi})} \delta \omega_i^+} = \norm{\pi_{p_{i+1}}\dexp_{\xi_{i+1}} \delta \xi_{i+1}}.
\]
It follows that
\[
\norm{\delta \omega_i^+} \le \norm{(\pi_{p_i} \dexp_{-\omega_i^+(\bvec{\xi})})^{-1}} \norm{\pi_{p_{i+1}} \dexp_{\xi_{i+1}}} \norm{\delta \xi_{i+1}},
\]
thus
\[\norm*{\frac{\partial \omega_i^+(\bvec{\xi})}{\partial
  \xi_{i+1}}} \le \norm{(\pi_{p_i} \dexp_{-\omega_i^+(\bvec{\xi})})^{-1}}
\norm{\pi_{p_{i+1}} \dexp_{\xi_{i+1}}}
.
\]
The claims in the proposition follow from \eqref{eq: propoffdiag1}  and \autoref{lem: dexpbound}.
\end{proof}

We note that we can also use a Taylor expansion of the right hand side to get the asymptotic bound
\[
\norm*{\frac{\partial \phi_i(\bvec{\xi})}{\partial \xi_{i+1}}} \le \frac{1}{4}+ \frac{1}{24} K^2 \norm{\omega_i^+(\bvec{\xi})}^2 + \frac{1}{24} K^2\norm{\xi_i}^2+ \mathcal{O}(K^4\norm{\omega_i^+(\bvec{\xi})}^4, K^4\norm{\bvec{\xi}}^4)
    .
\]

  To bound $\frac{\partial \phi_i(\bvec{\xi})}{\partial \xi_i}$, it is again advantageous to do the splitting
  $\phi_i(\bvec{\xi})= \phi_i^+(\bvec{\xi})+ \phi_i^-(\bvec{\xi})$, and consider each term separately.

  While it is possible to bound $\norm*{\frac{\partial \phi_i^+(\bvec{\xi})}{\partial \xi_i}}$ by trigonometric and hyperbolic functions of $K \norm{\xi_i}$ and $K \norm{\omega_i^+(\bvec{\xi})}$, the required calculations are lengthy, and we restrict ourself to an asymptotic bound.

  \begin{proposition}
    \label{prop: ondiag}
    Let $1\le i \le N-1$ and denote by $\frac{\partial \phi_i^+(\bvec{\xi})}{\partial \xi_i}\from
  \lie{m}_{p_i}\to \lie{m}_{p_i}$
  the partial derivative of $\phi_i$ with respect to $\xi_i$.
  In the operator norm
  \[
    \begin{aligned}
    \norm*{\frac{\partial \phi_i^+(\bvec{\xi})}{\partial \xi_i}}
    \le \frac{1}{4} K^2\left(\frac{1}{3} \norm{\omega_i^+(\bvec{\xi})}^2 + \frac{1}{6} \norm{\xi_i}^2 + \frac{1}{2} \norm{\xi_i}\norm{\omega_i^+(\bvec{\xi})}\right) \\
    + \mathcal{O}(K^4\norm{\omega_i^+(\bvec{\xi})}^4, K^4\norm{\xi_i}^4)
    \end{aligned}
  \]
  and
  \[
    \begin{aligned}
      \norm*{\frac{\partial \phi_i^-(\bvec{\xi})}{\partial \xi_i}} \le \frac{1}{4} K^2\left(\frac{1}{3} \norm{\omega_i^-(\bvec{\xi})}^2 + \frac{1}{6} \norm{\xi_i}^2 + \frac{1}{2} \norm{\xi_i}\norm{\omega_i^-(\bvec{\xi})}\right) \\
      + \mathcal{O}(K^4\norm{\omega_i^-(\bvec{\xi})}^4, K^4\norm{\xi_i}^4)
    \end{aligned}
  \]
  \end{proposition}
  Again, we will only prove the first claim, the proof of the second is entirely symmetrical.
  We first prove the following Lemma on the derivative of $\omega_i^+$.
  \begin{lemma}
    \label{prop:bounderomaga}
    For $1\le i\le N-1$, the following equality holds
  \[
  \frac{\partial \omega_i^+(\bvec{\xi})}{\partial \xi_i} \delta \xi_i= - (\pi_{p_i} \dexp_{-\omega_i^+(\bvec{\xi})})^{-1}  \pi_{p_i}  \Ad_{\exp(-\omega_i^+(\bvec{\xi}))}  \dexp_{-\xi_i}\delta \xi_i.
  \]
  \end{lemma}
  \begin{proof}
  We use that
  \begin{equation}
  \label{eq: omegaeq}
  \exp(\omega_i^+(\bvec{\xi})) \act p_i = \exp(-\xi_i) \act q_{i+1}^-
  .
  \end{equation}
  By differentiating implicitly, we get
  \[\begin{aligned}
  \exp(\omega_i^+(\bvec{\xi})) \act (\dexp_{-\omega_i^+(\bvec{\xi})} \delta \omega_i^+) \act p_i   &= - (\dexp_{-\xi_i} \delta \xi_i) \act \exp(-\xi_i) \act q_{i+1}^-
  \\&= - (\dexp_{-\xi_i} \delta \xi_i) \act \exp(\omega_i^+(\bvec{\xi})) \act p_i,
  \end{aligned}
  \]
  where the second equality uses \eqref{eq: omegaeq}.
  Acting from the left with $\exp(-\omega_i^+(\bvec{\xi}))$, we get
  \[
  (\dexp_{-\omega_i^+(\bvec{\xi})}\delta \omega_i^+(\bvec{\xi})) \act p_i   = - (\Ad_{\exp(-\omega_i^+(\bvec{\xi}))}  \dexp_{-\xi_i} \delta \xi)  \act  p_i
  .
  \]
  Applying the connection $\varpi$ to both sides, we get
  \[
  \left(\pi_{p_i}\dexp_{-\omega_i^+(\bvec{\xi})}\right) \delta \omega_i^+   = - \left(\pi_{p_i} \Ad_{\exp(-\omega_i^+(\bvec{\xi}))}  \dexp_{-\xi_i} \right)\delta \xi_i,
  \]
  where the bracketed expressions are linear operators $\lie{m}_{p_i} \to \lie{m}_{p_i}$.
  Now, the map ${\pi_{p_i}\dexp_{-\omega_i^+(\bvec{\xi})}} \from \lie{m}_{p_i} \to \lie{m}_{p_i}$ is invertible by \autoref{prop:adeven}, and the claim follows.
  \end{proof}

  We now proceed with the proof of \autoref{prop: ondiag}.
  \begin{proof}[Proof of \autoref{prop: ondiag}]
  From the definition of $\phi_i^+$,
  \[
  \frac{\partial \phi_i^+(\bvec{\xi})}{\partial \xi_i} (\delta \xi_i)= \frac{1}{4} \left(\delta \xi_i+\frac{\partial \omega_i^+(\bvec{\xi})}{\partial \xi_i}(\delta \xi_i) \right)
  \]

  Using \autoref{prop:bounderomaga}, we have
  \begin{equation}
  \frac{\partial \phi_i^+(\bvec{\xi})}{\partial \xi_i} \delta \xi_i = \frac{1}{4} \left (\delta \xi_i- (\pi_{p_i} \dexp_{-\omega_i^+(\bvec{\xi})})^{-1}  \pi_{p_i} \Ad_{\exp(-\omega_i^+(\bvec{\xi}))} \dexp_{-\xi_i}\delta \xi_i \right)
  .
  \end{equation}
  Using \autoref{prop:adeven}, the series expansions for $\Ad_{\exp(-\omega_i^+(\bvec{\xi}))} = \exp(-\ad_{\omega_i^+})$, and $\dexp_{-\xi_i}$,
  the expression on the right hand side can be expanded as an infinite series in $\ad_{\omega_i^+(\bvec{\xi})}$ and $\ad_{\xi_i}$ applied to $\delta \xi_i$.
  (Recall that, as we saw in the proof of \autoref{prop:adeven}, the effect of $\pi_{p_i}$ is to cancel out Lie monomials of even degree.)
  The crucial point is that there is a cancellation in the leading term so that, neglecting terms of fourth order and higher,
  \begin{equation}
    \begin{split}
      &\delta \xi_i- (\pi_{p_i}  \dexp_{-\omega_i^+(\bvec{\xi})})^{-1}  \pi_{p_i}  \Ad_{\exp(-\omega_i^+(\bvec{\xi}))}  \dexp_{-\xi_i} \delta \xi_i  =\\
      &\quad -\frac{1}{3} \ad_{\omega_i^+(\bvec{\xi})}^2(\delta \xi_i) - \frac{1}{6} \ad_{\xi_i}^2(\delta \xi_i) - \frac{1}{2}\ad_{\omega_i^+(\bvec{\xi})}\ad_{\xi_i}(\delta \xi_i) \\
      &\qquad + \mathcal{O}\paren[\Big]{K^4\norm{\omega_i^+(\bvec{\xi})}^4\norm{\delta \xi_i}, K^4\norm{\xi_i}^4\norm{\delta \xi_i}}
  .
  \end{split}
  \end{equation}

  We thus get
  \[
  \frac{\partial \phi_i^+(\bvec{\xi})}{\partial \xi_i} = -\frac{1}{4} \left(\frac{1}{3} \ad_{\omega_i^+(\bvec{\xi})}^2 + \frac{1}{6} \ad_{\xi_i}^2 + \frac{1}{2}\ad_{\omega_i^+(\bvec{\xi})}\ad_{\xi_i} + \mathcal{O}(K^4\norm{\omega_i^+(\bvec{\xi})}, K^4\norm{\xi_i}^4)\right)
  .
  \]
  The proposition follows by taking norms.
  \end{proof}

  It is also possible to bound the derivatives of the boundary condition functions.
  \begin{proposition}
  \label{prop: boundary}
  \mbox{}
  \begin{enumerate}[label=\upshape(\roman*)]
  \item
    (Clamped splines) Let $\phi_0$ and $\phi_N$ be as in
    \eqref{eq:clbc}. Then $\frac{\partial \phi_0(\bvec{\xi)})}{\partial \xi_1}=0$, $\frac{\partial \phi_N(\bvec{\xi})}{\partial \xi_{N-1}}=0$.
  \item
   (Natural splines) Let $\phi_0$ and $\phi_N$ be as on \eqref{eq:frbc}. Then
   \[\norm*{\frac{\partial \phi_0(\bvec{\xi})}{\partial \xi_1}} \le \frac{1}{2} + \frac{1}{3}
    K^2\norm{\phi_0}^2 + \frac{1}{12}K^2 \norm{\xi_1}^2 + \mathcal{O}(K^4
    \norm{\bvec{\xi}}^4)\]
  and similar for $\norm*{\frac{\partial \phi_N(\bvec{\xi})}{\partial \xi_{N-1}}}$.
\end{enumerate}

  \end{proposition}
  \begin{proof}
    \mbox{}
  \begin{enumerate}[label=\upshape(\roman*)]
  \item
    For clamped splines, $\phi_0$ and $\phi_N$ are constant.
  \item
    For natural splines, we have
  \[
    \exp(2\phi_0(\bvec{\xi}))\act p_0 = \exp(-\xi_1)\act p_1
    .
  \]
  Differentiating and taking norms, we get
  \[
  2\norm{\pi_{p_0} \dexp_{-2\phi_0} \delta \phi_0} = \norm{\pi_{p_1}
    \dexp_{\xi_1} \delta \xi_1}
  ,
  \]
  where $\delta \phi_0= \frac{\partial \phi_0(\bvec{\xi})}{\partial
    \xi_1}\delta \xi_1$.
  \autoref{lem: dexpbound} now gives
  \[ \norm{\delta \phi_0} \le \frac{1}{2}\frac {2K\norm{\phi_0}}{\sin(2K \norm
      {\phi_0})} \frac{ \sinh(K \norm{\xi_1})}{K\norm{\xi_1}}\norm{\delta \xi_1}
  ,
  \]
  and the claim follows by a Taylor expansion of $\frac{x}{\sin(x)}$ and
  $\frac{\sinh(x)}{x}$.
  The corresponding bound and proof for $\norm*{\frac{\partial
      \phi_N(\bvec{\xi})}{\partial \xi_{N-1}}}$ is entirely symmetrical.
  \end{enumerate}
  \end{proof}

  \subsection{Convergence}\label{sub:convergence_proof_final}

  We are now ready to prove the convergence theorem.

  \begin{proof}[Proof of \autoref{thm:main_result_convergence}]
  A consequence of \autoref{prop: invarea} is that $\bvec{\phi}$ has an invariant
  region for small enough $D$. We need to also establish that it is a
  contraction. It is sufficient to show that there exists an $\alpha \in (0,1)$
  such that $\norm{\Tan_{\bvec{\xi}} \bvec{\phi}}\le \alpha$ in an invariant region.

  We have
  \[\norm{\Tan_{\bvec{\xi}}\bvec{\phi}} = \max_{0\le i \le N} \norm{\Tan_{\bvec{\xi}} \phi_i} \]
  as well as
  \[\begin{aligned}
      \norm{\Tan_{\bvec{\xi}} \phi_0} &= \norm*{\frac{\partial \phi_0(\bvec{\xi})}{\partial \xi_1}}
      ,
      \\
      \norm{\Tan_{\bvec{\xi}} \phi_i} &= \norm*{\frac{\partial \phi_i(\bvec{\xi})}{\partial \xi_{i-1}}} + \norm*{\frac{\partial \phi_i(\bvec{\xi})}{\partial \xi_{i}}} + \norm*{\frac{\partial \phi_i(\bvec{\xi})}{\partial \xi_{i+1}}}
      ,
      \\
      \norm{\Tan_{\bvec{\xi}} \phi_N} &= \norm*{\frac{\partial \phi_N(\bvec{\xi})}{\partial \xi_{N-1}}}
      .
  \end{aligned}
  \]

  By \autoref{prop: offdiag}, \autoref{prop: ondiag}, and \autoref{prop: boundary}, we can bound the terms appearing with a Taylor expansion to obtain
  \[\begin{aligned}
  \norm{\Tan_{\bvec{\xi}} \phi_i} \le
  &\frac{1}{4}  \left(1+ \frac{1}{6} K^2 \norm{\omega_i^-(\bvec{\xi})}^2 + \frac{1}{6} K^2\norm{\xi_i}^2\right) \\
  &+ \frac{1}{4} K^2\left(\frac{1}{3} \norm{\omega_i^-(\bvec{\xi})}^2 + \frac{1}{6} \norm{\xi_i}^2 + \frac{1}{2} \norm{\xi_i}\norm{\omega_i^-(\bvec{\xi})}\right) \\
  &+ \frac{1}{4} K^2\left(\frac{1}{3} \norm{\omega_i^+(\bvec{\xi})}^2 + \frac{1}{6} \norm{\xi_i}^2 + \frac{1}{2} \norm{\xi_i}\norm{\omega_i^+(\bvec{\xi})}\right) \\
  &+ \frac{1}{4}  \left(1+ \frac{1}{6} K^2 \norm{\omega_i^+(\bvec{\xi})}^2 + \frac{1}{6} K^2\norm{\xi_i}^2\right)+ \mathcal{O}(K^4\norm{\bvec{\omega}}^4, K^4\norm{\bvec{\xi}}^4)\\
  \le & \frac{1}{2} + \frac{1}{4} K^2 \norm{\bvec{\omega}}^2 + \frac{1}{6} K^2 \norm{\bvec{\xi}}^2 + \frac{1}{4} K^2\norm{\bvec{\xi}}\norm{\bvec{\omega}} + \mathcal{O}(K^4 \norm{\bvec{\omega}}^4, K^4\norm{\bvec{\xi}}^4).
  \end{aligned}
  \]
  We can now use \eqref{eq:omegabound} and \autoref{prop: boundary} to obtain a bound of the form
  \[
  \norm{\Tan_{\bvec{\xi}} \bvec{\phi}} \le \frac{1}{2} + \mathcal{O}(K^2D^2, K^2\norm{\bvec{\xi}}^2).
  \]
  It is therefore clear that when $K\max(D, \norm{\bvec{\xi}})$ is sufficiently small, $\bvec{\phi}$ is a contraction.
  When $D$ approaches zero, the smallest $V$ satisfying  \eqref{eq: Vineq} also approaches zero.
  Therefore, when $D$ is small enough, there is a $V$ such that
  \begin{itemize}
  \item $\setc[\Big]{ \bvec{\xi}}{\norm{\bvec{\xi}}\le V}$ is a compact invariant region and
  \item $\bvec{\phi}$ is a contraction on that region.
  \end{itemize}
  Therefore, the fixed point iteration converges to a solution.
  \end{proof}
  See \autoref{fig:proofplot} for an illustration of the final argument in the proof.

  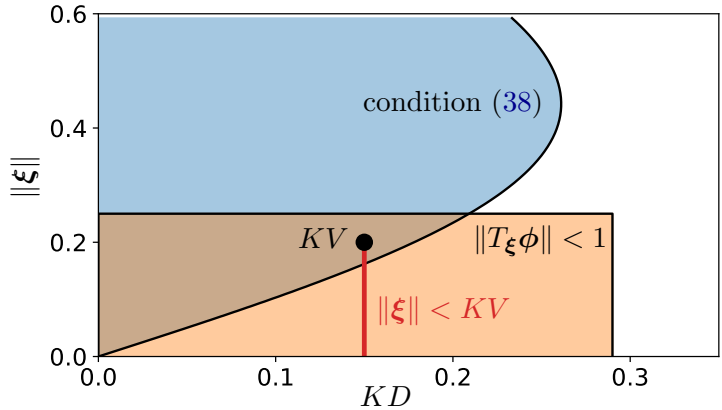
\begin{figure}
    \centering
    \begin{tikzpicture}
      \node[anchor=south west, inner sep=0] (image) at (0,0) {\includegraphics[width=0.6\textwidth]{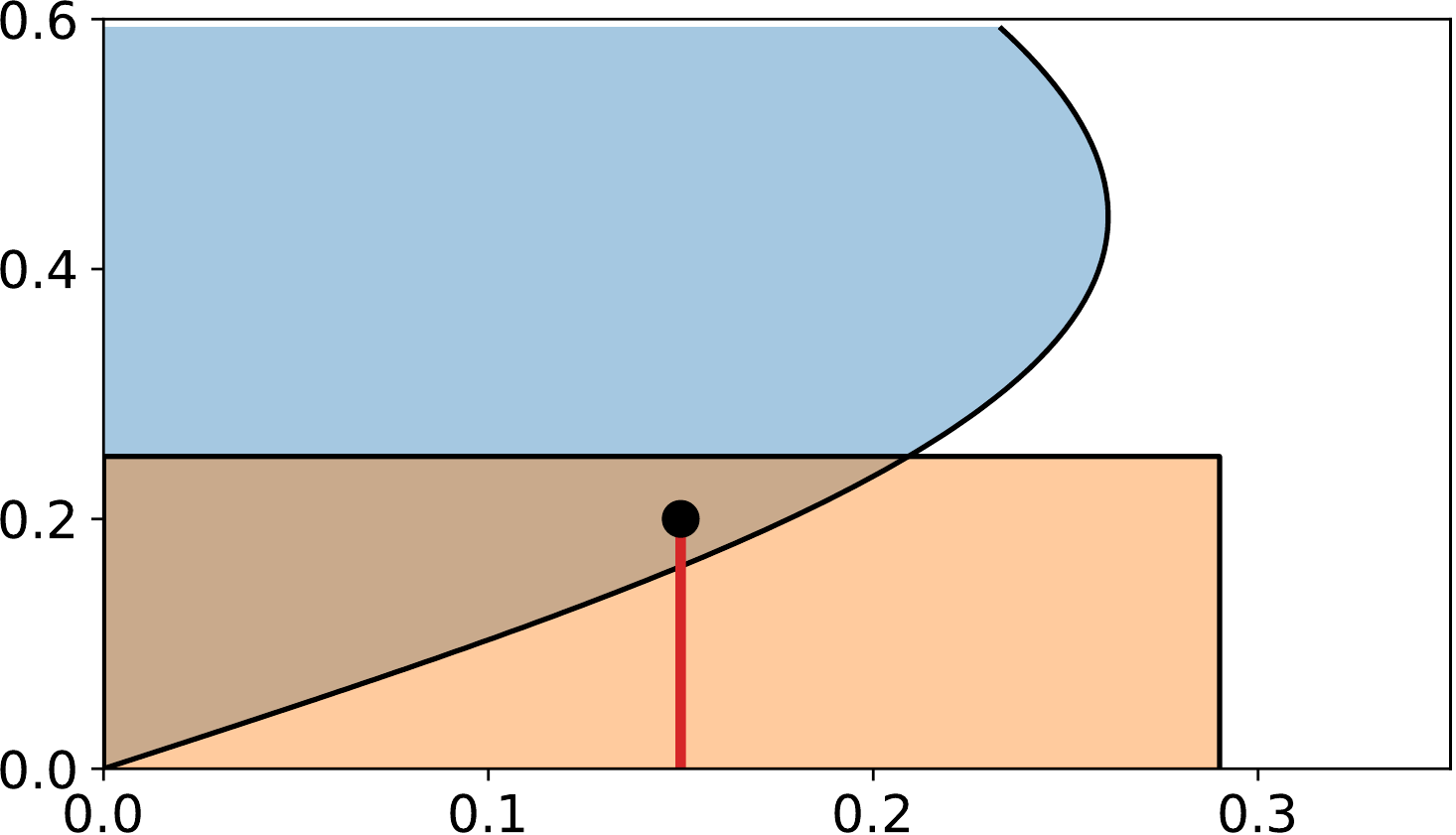}};
      \begin{scope}[x={(image.south east)},y={(image.north west)}]
        \coordinate (xlabel) at (0.5,0.03) {};
        \coordinate (ylabel) at (0.00,0.55) {};
        \coordinate (KVlabel) at (0.46,0.39) {};
        \coordinate (ineq1label) at (0.47,0.2) {};
        \coordinate (ineq2label) at (0.845,0.45) {};
        \coordinate (conditionlabel) at (0.75,0.74) {};
        \node[left] at (conditionlabel) {\color[HTML]{000000}condition \eqref{eq: Vineq}};
        \node[below left, rotate=0] at (ineq2label) {\color[HTML]{000000}$\norm{\Tan_{\bvec{\xi}} \bvec{\phi}}<1$};
        \node[right] at (ineq1label) {\color[HTML]{d62728}$\norm{\bvec{\xi}}<KV$};
        \node[left] at (KVlabel) {$KV$};
        \node[below] at (xlabel) {$KD$};
        \node[above,rotate=90] at (ylabel) {$\norm{\bvec{\xi}}$};
      \end{scope}
    \end{tikzpicture}
    \caption{An illustration of the convergence argument. 
      The upper curved region (light blue) denotes the values $(KD, KV)$ satisfying \eqref{eq: Vineq}.
      The lower rectangular region (light orange) illustrates the condition $K\max(D, \norm{\bvec{\xi}})$ sufficiently small.
      ``Sufficiently small'' is not quantified, and the region is just an illustration.
      To ensure a solution, $KV$ has to be in the upper curved region, and the line below $KV$ has to be contained in the lower rectangular region.
    }
    \label{fig:proofplot}
  \end{figure}

  \section{Python implementation}
  \label{sec:python}

  Here we briefly indicate how to use our \texttt{Python} implementation to compute splines in various geometries.
  First, one has to install the following package by following the installation instructions:
  \begin{center}
    \url{https://github.com/olivierverdier/bsplinelab}
  \end{center}

  The minimal necessary import is
  \begin{python}
  from bspline.interpolation import Symmetric, cubic_spline
  \end{python}

  Here is an example of how to interpolate between random points in the Euclidean case:
  \begin{python}
  # 10 random points in R^4:
  interpolation_points = np.random.rand(10,4) 
  b = cubic_spline(Symmetric, interpolation_points)
  \end{python}
  The result is then a $C^2$ function \pyth!b! which is equal to the interpolation  points at integer points on the interval $[0,9]$.

  In order to interpolate on, for instance, a sphere, one would do as follows:
  \begin{python}
  from bspline.geometry import Sphere
  # 3 points on the sphere:
  interpolation_points = np.array([[1.,0,0], [0,1,0], [0,0,1]])
  b = cubic_spline(Symmetric, interpolation_points, geometry=Sphere())
  \end{python}
  The function \pyth!b! is now a function defined on the interval $[0,2]$ such that it is $C^2$, is interpolating the prescribed points at the integer points $0$, $1$ and $2$, and is on the sphere for all points in between.

  For further information on \texttt{bsplinelab}, we refer to the package documentation and the available code examples.

  \bibliographystyle{amsplainnat} 
  \bibliography{riemann}

  \end{document}